\documentclass{amsart}
\usepackage{amsfonts}
\usepackage{amsmath}

\newtheorem{theorem}{Theorem}[section]

\newtheorem{lemma}[theorem]{Lemma}
\theoremstyle{remark}

\theoremstyle{definition}

\theoremstyle{proposition}
\newtheorem{proposition}[theorem]{Proposition}
\numberwithin{equation}{section}
 
\begin{document}
\title{Geometry of shrinking Ricci solitons}
\author{Ovidiu Munteanu and Jiaping Wang}

\begin{abstract}
The main purpose of this paper is to investigate the curvature behavior of
four dimensional shrinking gradient Ricci solitons. For such a soliton $M$
with bounded scalar curvature $S,$ it is shown that the curvature operator $%
\mathrm{Rm}$ of $M$ satisfies the estimate $|\mathrm{Rm}|\leq c\,S$ for some
constant $c.$ Moreover, the curvature operator $\mathrm{Rm}$ is
asymptotically nonnegative at infinity and admits a lower bound $\mathrm{Rm}%
\geq -c\,\left( \ln r\right) ^{-1/4},$ where $r$ is the distance function to
a fixed point in $M.$ As a separate issue, a diameter upper bound for
compact shrinking gradient Ricci solitons of arbitrary dimension is derived
in terms of the injectivity radius.
\end{abstract}

\maketitle


This paper primarily concerns the geometry of the so-called shrinking
gradient Ricci solitons. Recall that a complete manifold $(M,g)$ is a
gradient Ricci soliton if the equation

\begin{equation*}
\mathrm{Ric}+\mathrm{Hess}\left( f\right) =\lambda\,g
\end{equation*}
holds for some function $f$ and scalar $\lambda.$ Here, $\mathrm{Ric}$ is
the Ricci curvature of $\left( M,g\right) $ and $\mathrm{Hess}\left(
f\right) $ the Hessian of $f.$ Note that if the potential function $f$ is
constant or the soliton is trivial, then the soliton equation simply says
the Ricci curvature is constant. So Ricci solitons are natural
generalization of Einstein manifolds. A soliton is called shrinking, steady
and expanding, accordingly, if $\lambda>0,$ $\lambda=0$ and $\lambda<0.$ By
scaling the metric $g,$ one customarily assumes $\lambda \in \left\{
-1/2,0,1/2\right\}.$ Solitons may be regarded as self-similar solutions to
the Ricci flows. As such, they are important in the singularity analysis of
Ricci flows. Indeed, according to \cite{EMT}, the blow-ups around a type-I
singularity point always converge to nontrivial gradient shrinking Ricci
solitons. It is thus a central issue in the study of Ricci flows to
understand and classify gradient Ricci solitons.

Aside from the Einstein manifolds, the Euclidean space $\mathbb{R}^n$
together with potential function $f(x)=\frac{\lambda}{2}\,|x|^2$ gives
another important example of gradient Ricci solitons. In the case dimension $%
n=2,$ according to \cite{H}, those are the only examples. For dimension $%
n=3, $ Perelman made the breakthrough in \cite{P} and showed that a three
dimensional non-collapsing shrinking gradient Ricci soliton with bounded
curvature must be a quotient of the sphere $S^3,$ or $\mathbb{R}^3,$ or $%
S^2\times \mathbb{R}.$ His result played a crucial role in the affirmative
resolution of the Poincar\'e conjecture. The extra conditions were later
removed through the effort of Naber \cite{N}, Ni and Wallach \cite{NW}, and
Cao, Chen and Zhu \cite{CCZ}. We should refer the readers to \cite{B} for
the classification of steady gradient Ricci solitons.

One salient feature of three dimensional shrinking Ricci solitons is that
their curvature operator must be nonnegative \cite{H}. This has been of
great utility in Perelman's argument. Unfortunately, for dimension four or
higher, this is no longer true as demonstrated by the example constructed in 
\cite{FIK}. Also, the existence of examples (see \cite{Ca1} for a list)
other than the aforementioned ones complicates the classification outlook.

The main purpose here is to investigate the curvature behavior of four
dimensional shrinking gradient Ricci solitons. Our first result concerns the
control of the curvature operator. Note that in the case of dimension three,
the curvature operator, being nonnegative, is obviously bounded by the
scalar curvature. In the case of dimension four, while the curvature
operator no longer has a fixed sign, we show that such a conclusion still
holds. In particular, it implies that the curvature operator must be bounded.

\begin{theorem}
\label{Main}Let $\left( M,g,f\right) $ be a four dimensional shrinking
gradient Ricci soliton with bounded scalar curvature $S$. Then there exists
a constant $c>0$ so that 
\begin{equation*}
\left\vert \mathrm{Rm}\right\vert \leq c\,S\text{ \ on }M\text{.}
\end{equation*}
\end{theorem}

Our second result provides a lower bound for the curvature operator of a
four dimensional shrinking Ricci soliton with bounded scalar curvature. It
shows that the curvature operator becomes asymptotically nonnegative at
infinity. The result may be viewed as an extension of Hamilton and Ivey
curvature pinching estimate for the three dimension case. Note that Naber 
\cite{N} has classified all four dimensional shrinking Ricci solitons with
nonnegative and bounded curvature operator. In passing, we would also like
to point out that Cao and Chen \cite{CC} have obtained some interesting
classification results by imposing assumptions of different nature on the
curvature tensor.

\begin{theorem}
\label{Lower_Bound}Let $\left( M,g,f\right) $ be a four dimensional
shrinking gradient Ricci soliton with bounded scalar curvature. Then its
curvature operator is bounded below by 
\begin{equation*}
\mathrm{Rm}\geq -\left( \frac{c}{\ln r}\right) ^{\frac{1}{4}},
\end{equation*}
where $r$ is the distance function to a fixed point in $M.$
\end{theorem}

In both Theorem \ref{Main} and Theorem \ref{Lower_Bound}, the constant $c>0$
depends only on $A,$ the upper bound of the scalar curvature on $M,$ and $B,$
the maximum of curvature tensor $\left\vert \mathrm{Rm} \right\vert$ on a
geodesic ball $B_p(r_0),$ where $p$ is a minimum point of $f$ and $r_0$ is
determined by $A.$

We should point out that these conclusions are only effective for nontrivial
solitons. In fact, the potential function $f$ of the soliton is exploited in
an essential way in our proofs by working on the level sets of $f.$ Note
that the level set is of three dimension. So its curvature tensor is
determined by its Ricci curvature. This fact is crucial to our argument. It
enables us to control the curvature tensor of the ambient manifold by its
Ricci curvature, which leads to an estimate of the Ricci curvature by the
scalar curvature. The fact that Ricci curvature controls the growth of the
full curvature tensor is already known from the work of the first author and
Wang \cite{MWa}. However, the argument and estimate there are global in
nature, whereas the estimate here is valid in the pointwise sense, hence
stronger.

Our curvature estimate certainly leads to the conclusion that the set of
four dimensional shrinking gradient Ricci solitons with a fixed scalar
curvature upper bound must be compact in the orbifold sense. Moreover, the
possible orbifold points must be contained within a fixed compact set. In
this direction, we should point out that Haslhofer and M\"uller \cite{HM2}
have recently obtained an elegant and more general orbifold compactness
result for four dimensional shrinking gradient Ricci solitons, improving
various earlier results \cite{CS, HM1, Z, W}.

The fact that the curvature operator of four dimensional shrinking gradient
Ricci solitons enjoys similar control as in the dimension three case seems
to provide a glimpse of hope for a possible classification of the solitons.
In this sense, one goal here is to prove sharp decay estimates for the Riemann 
curvature tensor and its covariant derivatives, under the assumption that
the scalar curvature converges to zero at infinity.
This in particular enables us to conclude that such a soliton must in fact
be smoothly asymptotic to a cone at infinity. Here, by a cone, we mean a
manifold $[0,\infty )\times \Sigma $ endowed with Riemannian metric $%
g_{c}=dr^{2}+r^{2}\,g_{\Sigma },$ where $(\Sigma ,g_{\Sigma })$ is a closed $%
(n-1)$-dimensional Riemannian manifold. Denote $E_{R}=(R,\infty )\times
\Sigma $ for $R\geq 0$ and define the dilation by $\lambda $ to be the map $%
\rho _{\lambda }:E_{0}\rightarrow E_{0}$ given by $\rho _{\lambda }(r,\sigma
)=(\lambda \,r,\sigma ).$

We say that a Riemannian manifold $(M,g)$ is $C^{k}$ asymptotic to the cone $%
(E_{0},g_{c})$ if, for some $R>0,$ there is a diffeomorphism $\Phi
:E_{R}\rightarrow M\setminus \Omega $ such that $\lambda ^{-2}\,\rho
_{\lambda }^{\ast }\,\Phi ^{\ast }\,g\rightarrow g_{c}$ as $\lambda
\rightarrow \infty $ in $C_{loc}^{k}(E_{0},g_{c}),$ where $\Omega $ is a
compact subset of $M.$

We have the following result. 

\begin{theorem}
\label{cone_i} Let $\left( M,g,f\right) $ be a complete four
dimensional shrinking gradient Ricci soliton with scalar curvature
converging to zero at infinity. Then there exists a cone $E_0$ such that $%
(M,g)$ is $C^k$ asymptotic to $E_0$ for all $k.$
\end{theorem}

A recent result due to Kotschwar and L. Wang \cite{KW} states that two
shrinking gradient Ricci solitons must be isometric if they are $C^2$
asymptotic to the same cone. Together with our result, this implies that the
classification problem for four dimensional shrinking Ricci solitons with
scalar curvature going to zero at infinity is reduced to the one for the
limiting cones.

As a separate issue, we have also attempted to address the question whether
the limit of compact shrinking gradient Ricci solitons remains compact. This
question may be rephrased into one of obtaining a uniform upper bound for
the diameter of such solitons. Note that in the opposite direction Futaki
and Sano \cite{FS}, see also an improvement in \cite{FLL}, have already
established a universal diameter lower bound for (nontrivial) compact
shrinking gradient Ricci solitons. It remains to be seen whether a universal
diameter upper bound is available without any extra assumptions.

\begin{theorem}
\label{Diam}Let $\left( M,g,f\right) $ be a compact gradient shrinking Ricci
soliton of dimension $n$. Then the diameter of $\left( M,g\right) $ has an
upper bound of the form 
\begin{equation*}
\mathrm{diam}\left( M\right) \leq c\left( n,\mathrm{inj}\left( M\right)
\right) ,
\end{equation*}%
where $\mathrm{inj}\left( M\right) $ is the injectivity radius of $\left(
M,g\right) .$
\end{theorem}

If one assumes in addition that the Ricci curvature of the soliton is
bounded below, then the conclusion follows from \cite{FMZ}. We also remark
that the assumption on the injectivity radius seems to be natural in view of
the non-collapsing result for Ricci flows proved by Perelman \cite{P}.
Certainly, our result also implies an upper bound for the volume, depending
on the injectivity radius alone.

\section{Curvature estimates}

In this section, we show that the curvature operator of a four dimensional
shrinking gradient Ricci soliton must be bounded if its scalar curvature is
so. We first recall some general facts concerning shrinking gradient Ricci
solitons which will be used throughout the paper. For $(M^n, g, f)$ a
shrinking gradient Ricci soliton, it is known \cite{H} that $S+\left\vert
\nabla f\right\vert ^{2}-f$ is constant on $M$, where $S$ is the scalar
curvature of $M.$ So, by adding a constant to $f$ if necessary, we may
normalize the soliton such that

\begin{equation}
S+\left\vert \nabla f\right\vert ^{2}=f.  \label{s}
\end{equation}%
Also, by a result of Chen \cite{Ch, Ca1}, the scalar curvature $S>0$ unless $%
M$ is flat. So in the following, we will assume without loss of generality
that $S>0.$ Tracing the soliton equation we get $\Delta f+S=\frac{n}{2}.$
Combined with (\ref{s}), this implies that 
\begin{equation}
\Delta _{f}\left( f\right) =\frac{n}{2}-f.  \label{delta}
\end{equation}
Here, $\Delta _{f}=\Delta -\left\langle \nabla f,\nabla \right\rangle $ is
the weighted Laplacian on $M,$ which is self adjoint on the space of square
integrable functions with respect to the weighted measure $e^{-f}dv.$

Concerning the potential function $f,$ Cao and Zhou \cite{CZ} have proved
that

\begin{equation}
\left( \frac{1}{2}r\left( x\right) -c\right) ^{2}\leq f\left( x\right) \leq
\left( \frac{1}{2}r\left( x\right) +c\right) ^{2}  \label{f}
\end{equation}%
for all $r\left(x\right) \geq r_{0}.$ Here $r\left( x\right) :=d\left(
p,x\right)$ is the distance of $x$ to $p,$ a minimum point of $f$ on $M,$
which always exists. Both constants $r_{0}$ and $c$ can be chosen to depend
only on dimension $n.$

Throughout the paper, we denote 
\begin{eqnarray*}
D\left( t\right) &=&\left\{ x\in M:f\left( x\right) \leq t\right\} \\
\Sigma \left( t\right) &=&\partial D\left( t\right) =\left\{ x\in M:f\left(
x\right) =t\right\} .
\end{eqnarray*}%
By (\ref{f}), these are compact subsets of $M$.

We also recall the following equations for curvatures. For proofs, one may
consult \cite{PW}. Here we follow the notations in \cite{MWa}.

\begin{eqnarray}
\Delta _{f}S &=&S-2\left\vert \mathrm{Ric}\right\vert ^{2}  \label{id} \\
\Delta _{f}R_{ij} &=&R_{ij}-2R_{ikjl}R_{kl}  \notag \\
\Delta _{f}\mathrm{Rm} &=&\mathrm{Rm}+\mathrm{Rm}\ast \mathrm{Rm}  \notag \\
\nabla _{k}R_{jk} &=&R_{jk}f_{k}=\frac{1}{2}\nabla _{j}S  \notag \\
\nabla _{l}R_{ijkl} &=&R_{ijkl}f_{l}=\nabla _{j}R_{ik}-\nabla _{i}R_{jk}. 
\notag
\end{eqnarray}%
In this section we will assume 
\begin{equation}
S\leq A\text{ }\ \text{on }M  \label{*}
\end{equation}%
for some constant $A>0.$ Obviously, there exists $r_{0}>0$, depending only
on $A$, so that 
\begin{equation*}
\left\vert \nabla f\right\vert \geq \frac{1}{2}\sqrt{f}\geq 1\text{ \ on \ }%
M\backslash D\left( r_{0}\right) .
\end{equation*}%
Our argument is based on the following important observation.

\begin{proposition}
\label{Curv}Let $\left( M,g,f\right) $ be a four dimensional shrinking
gradient Ricci soliton. Then for a universal constant $c>0,$

\begin{equation}
\left\vert \mathrm{Rm}\right\vert \leq c\left( \frac{\left\vert \nabla 
\mathrm{Ric}\right\vert }{\sqrt{f}}+\frac{\left\vert \mathrm{Ric}\right\vert
^{2}+1}{f}+\left\vert \mathrm{Ric}\right\vert \right)  \label{C0}
\end{equation}%
on $M\backslash D\left( r_{0}\right).$
\end{proposition}

\begin{proof}
We work on $\Sigma :=\Sigma \left( t\right) ,$ $t\geq r_{0}.$ By the Gauss
curvature equation, for an orthonormal frame $\left\{
e_{1},e_{2},e_{3}\right\} $ tangent to $\Sigma ,$ the intrinsic Riemann
curvature tensor $R_{abcd}^{\Sigma }$ of $\Sigma $ is given by

\begin{equation}
R_{abcd}^{\Sigma }=R_{abcd}+h_{ac}h_{bd}-h_{ad}h_{bc},  \label{C1}
\end{equation}%
where $R_{abcd}=\mathrm{Rm}\left( e_{a},e_{b},e_{c},e_{d}\right) $ is the
Riemann curvature tensor of $M$ and $h_{ab}$ the second fundamental form of $%
\Sigma $. In what follows, the indices $a,b,c,d \in \{1,2,3\}$ and $i,j,k,
l\in \{1,2,3,4\}.$ Since $\Sigma =\left\{ f=t\right\} ,$ we have

\begin{equation*}
h_{ab}=\frac{f_{ab}}{\left\vert \nabla f\right\vert }.
\end{equation*}%
Using the fact that $\left\vert \nabla f\right\vert \geq \frac{1}{2}\sqrt{f}$
on $M\backslash D\left( r_{0}\right) $ and $\mathrm{Ric}(e_{a},e_{b})+f_{ab}=%
\frac{1}{2}\delta _{ab},$ we have

\begin{equation}
\left\vert h_{ab}\right\vert \leq \frac{c}{\sqrt{f}}\left( \left\vert 
\mathrm{Ric}\right\vert +1\right) .  \label{C2}
\end{equation}%
Since $\Sigma $ is a three dimensional manifold, its Riemann curvature is
determined by its Ricci curvature $\mathrm{Ric}^{\Sigma }.$ 
\begin{eqnarray}
R_{abcd}^{\Sigma } &=&\left( R_{ac}^{\Sigma }g_{bd}-R_{ad}^{\Sigma
}g_{bc}+R_{bd}^{\Sigma }g_{ac}-R_{bc}^{\Sigma }g_{ad}\right)  \label{C3} \\
&&-\frac{S^{\Sigma }}{2}\left( g_{ac}g_{bd}-g_{ad}g_{bc}\right),  \notag
\end{eqnarray}%
where $S^{\Sigma }$ is the scalar curvature of $\Sigma $ and $R_{ab}^{\Sigma
}:=\mathrm{Ric}^{\Sigma }\left( e_{a},e_{b}\right) $.

By tracing (\ref{C1}) we get

\begin{equation}
R_{ac}^{\Sigma }=R_{ac}-R_{a4c4}+H\,h_{ac}-h_{ab}\,h_{bc},  \label{C4}
\end{equation}
where $R_{a4c4}=\mathrm{Rm}\left( e_{a},\nu ,e_{c},\nu \right) $ with $\nu =%
\frac{\nabla f}{\left\vert \nabla f\right\vert }$ being the normal vector of 
$\Sigma $. Tracing this one more time, we have

\begin{equation*}
S^{\Sigma }=S-2R_{44}+H^{2}-\left\vert h\right\vert ^{2}.
\end{equation*}

Hence, by (\ref{C2}) and $S\leq 2\,\left\vert \mathrm{Ric}\right\vert ,$ we
see that

\begin{equation*}
\left\vert S^{\Sigma }\right\vert \leq c\left( \frac{\left\vert \mathrm{Ric}%
\right\vert ^{2}+1}{f}+\left\vert \mathrm{Ric}\right\vert \right)
\end{equation*}%
for some constant $c$. We now observe that by (\ref{id}) we have an estimate

\begin{equation}
\left\vert R_{ijk4}\right\vert =\frac{1}{\left\vert \nabla f\right\vert }%
\left\vert R_{ijkl}f_{l}\right\vert \leq 4\frac{\left\vert \nabla \mathrm{Ric%
}\right\vert }{\sqrt{f}}.  \label{C5}
\end{equation}

Using this in (\ref{C4}) implies that 
\begin{equation*}
\left\vert R_{ac}^{\Sigma }\right\vert \leq c\left( \frac{\left\vert \nabla 
\mathrm{Ric}\right\vert }{\sqrt{f}}+\frac{\left\vert \mathrm{Ric}\right\vert
^{2}+1}{f}+\left\vert \mathrm{Ric}\right\vert \right) .
\end{equation*}%
Hence, we conclude from (\ref{C3}) and (\ref{C1}) that 
\begin{equation*}
\left\vert R_{abcd}\right\vert \leq c\left( \frac{\left\vert \nabla \mathrm{%
Ric}\right\vert }{\sqrt{f}}+\frac{\left\vert \mathrm{Ric}\right\vert ^{2}+1}{%
f}+\left\vert \mathrm{Ric}\right\vert \right) .
\end{equation*}%
Together with (\ref{C5}), this proves the proposition.
\end{proof}

We now establish the following lemma. It is inspired by Hamilton's work in 
\cite{H1}.

\begin{lemma}
\label{DI}Let $\left( M,g,f\right) $ be a four dimensional shrinking
gradient Ricci soliton with bounded scalar curvature. Then, for any $0<a<1,$
the function $u:=\left\vert \mathrm{Ric}\right\vert ^{2}S^{-a}$ verifies the
differential inequality 
\begin{equation*}
\Delta _{f}u\geq \left( 2a-\frac{c}{1-a}\frac{S}{f}\right)
u^{2}S^{a-1}-c\,u^{\frac{3}{2}}\,S^{\frac{a}{2}}-c\,u
\end{equation*}%
on $M\backslash D\left( r_{0}\right) ,$ for some $r_{0}>0$ which \ depends
only on $A$. Here, $c>0$ is a universal constant .
\end{lemma}

\begin{proof}
Note that by Proposition \ref{Curv} and (\ref{id}) 
\begin{eqnarray}
\Delta _{f}\left\vert \mathrm{Ric}\right\vert ^{2} &\geq &2\left\vert \nabla 
\mathrm{Ric}\right\vert ^{2}-c\left\vert \mathrm{Rm}\right\vert \left\vert 
\mathrm{Ric}\right\vert ^{2}  \label{F1} \\
&\geq &2\left\vert \nabla \mathrm{Ric}\right\vert ^{2}-\frac{c}{\sqrt{f}}%
\left\vert \nabla \mathrm{Ric}\right\vert \left\vert \mathrm{Ric}\right\vert
^{2}  \notag \\
&&-\frac{c}{f}\left\vert \mathrm{Ric}\right\vert ^{4}-c\left\vert \mathrm{Ric%
}\right\vert ^{3}-\frac{c}{f}\left\vert \mathrm{Ric}\right\vert ^{2}.  \notag
\end{eqnarray}%
For $1>a>0$ direct computation gives

\begin{eqnarray}
&&\Delta _{f}\left( \left\vert \mathrm{Ric}\right\vert ^{2}S^{-a}\right)
\label{F2} \\
&=&S^{-a}\Delta _{f}\left( \left\vert \mathrm{Ric}\right\vert ^{2}\right)
+\left\vert \mathrm{Ric}\right\vert ^{2}\Delta _{f}\left( S^{-a}\right)
+2\left\langle \nabla S^{-a},\nabla \left\vert \mathrm{Ric}\right\vert
^{2}\right\rangle  \notag \\
&=&S^{-a}\Delta _{f}\left( \left\vert \mathrm{Ric}\right\vert ^{2}\right)
+2\left\langle \nabla S^{-a},\nabla \left\vert \mathrm{Ric}\right\vert
^{2}\right\rangle  \notag \\
&+&\left\vert \mathrm{Ric}\right\vert ^{2}\left( -aS^{-a}+2a\left\vert 
\mathrm{Ric}\right\vert ^{2}S^{-a-1}+a\left( a+1\right) \left\vert \nabla
S\right\vert ^{2}S^{-a-2}\right).  \notag
\end{eqnarray}

We can estimate 
\begin{eqnarray*}
2\left\langle \nabla S^{-a},\nabla \left\vert \mathrm{Ric}\right\vert
^{2}\right\rangle &\geq &-4a\left\vert \nabla \mathrm{Ric}\right\vert
\left\vert \nabla S\right\vert S^{-a-1}\left\vert \mathrm{Ric}\right\vert \\
&\geq &-a\left( a+1\right) \left\vert \nabla S\right\vert
^{2}S^{-a-2}\left\vert \mathrm{Ric}\right\vert ^{2} \\
&&-\frac{4a}{a+1}\left\vert \nabla \mathrm{Ric}\right\vert ^{2}S^{-a}.
\end{eqnarray*}%
Plugging this in (\ref{F2}) and combining with (\ref{F1}) shows that 
\begin{eqnarray*}
\Delta _{f}\left( \left\vert \mathrm{Ric}\right\vert ^{2}S^{-a}\right) &\geq
&\frac{2\left( 1-a\right) }{1+a}\left\vert \nabla \mathrm{Ric}\right\vert
^{2}S^{-a}-\frac{c}{\sqrt{f}}\left\vert \nabla \mathrm{Ric}\right\vert
\left\vert \mathrm{Ric}\right\vert ^{2}S^{-a}-\frac{c}{f}\left\vert \mathrm{%
Ric}\right\vert ^{4}S^{-a} \\
&&-c\left\vert \mathrm{Ric}\right\vert ^{3}S^{-a}-\frac{c}{f}\left\vert 
\mathrm{Ric}\right\vert ^{2}S^{-a}-a\left\vert \mathrm{Ric}\right\vert
^{2}S^{-a}+2a\left\vert \mathrm{Ric}\right\vert ^{4}S^{-a-1} \\
&\geq &\left( 2a-\frac{c}{1-a}\frac{S}{f}\right) \left\vert \mathrm{Ric}%
\right\vert ^{4}S^{-a-1}-c\left\vert \mathrm{Ric}\right\vert
^{3}S^{-a}-c\left\vert \mathrm{Ric}\right\vert ^{2}S^{-a}.
\end{eqnarray*}%
In the last line, we have used that

\begin{equation*}
\frac{c}{\sqrt{f}}\left\vert \nabla \mathrm{Ric}\right\vert \left\vert 
\mathrm{Ric}\right\vert ^{2}S^{-a}\leq \frac{2\left( 1-a\right) }{1+a}%
\left\vert \nabla \mathrm{Ric}\right\vert ^{2}S^{-a}+\frac{1+a}{8\left(
1-a\right) }\frac{c^2}{f}\left\vert \mathrm{Ric}\right\vert ^{4}S^{-a}.
\end{equation*}
It follows that

\begin{equation*}
\Delta _{f}u\geq \left( 2a-\frac{c}{1-a}\frac{S}{f}\right) u^{2}S^{a-1}
-c\,u^{\frac{3}{2}}S^{\frac{a}{2}}-c\,u.
\end{equation*}%
This proves the result.
\end{proof}

\begin{proposition}
\label{Rc}Let $\left( M,g,f\right) $ be a four dimensional shrinking
gradient Ricci soliton with bounded scalar curvature $S\le A.$ Then

\begin{equation*}
\sup_{M}\frac{\left\vert \mathrm{Ric}\right\vert ^{2}}{S}\leq C
\end{equation*}%
for a constant $C>0$ depending only on $A$ and $\sup_{D\left( r_{0}\right)
}\left\vert \mathrm{Rm}\right\vert $.
\end{proposition}

\begin{proof}
We use Lemma \ref{DI} that

\begin{equation*}
\Delta _{f}u\geq \left( 2a-\frac{c}{1-a}\frac{S}{f}\right)
u^{2}S^{a-1}-c\,u^{\frac{3}{2}}S^{\frac{a}{2}}-c\,u
\end{equation*}%
on $M\backslash D\left( r_{0}\right),$ where $u:=\left\vert \mathrm{Ric}%
\right\vert ^{2}S^{-a}$ and $0<a<1.$

For $R>2r_{0},\,$let $\phi $ be a smooth non-negative function defined on
the real line so that $\phi \left( t\right) =1$ for $R\leq t\leq 2R$ and $%
\phi \left( t\right) =0$ for $t\leq \frac{R}{2}$ and for $t\geq 3R$. We may
choose $\phi $ so that 
\begin{equation*}
t^{2}\left( \left\vert \phi ^{\prime }\right\vert ^{2}\left( t\right)
+\left\vert \phi ^{\prime \prime }\right\vert \left( t\right) \right) \leq c.
\end{equation*}%
We use $\phi \left( f\left( x\right) \right) $ as a cut-off function with
support in $D\left( 3R\right) \backslash D\left( \frac{R}{2}\right) .$ Note
that 
\begin{eqnarray*}
\left\vert \nabla \phi \right\vert &\leq &\frac{c}{\sqrt{R}} \\
\left\vert \Delta _{f}\phi \right\vert &\leq &c
\end{eqnarray*}
for a universal constant $c>0.$ Here, for the second inequality we have used
(\ref{delta}) that $\Delta _{f}\left( f\right) =2-f.$

Direct computation gives

\begin{eqnarray*}
\phi ^{2}\Delta _{f}\left( u\phi ^{2}\right) &=&\phi ^{4}\left( \Delta
_{f}u\right) +\phi ^{2}u\left( \Delta _{f}\phi ^{2}\right) +2\phi
^{2}\left\langle \nabla u,\nabla \phi ^{2}\right\rangle \\
&\geq &\left( 2a-\frac{c}{1-a}\frac{S}{f}\right) u^{2}S^{a-1}\phi ^{4}-cu^{%
\frac{3}{2}}S^{\frac{a}{2}}\phi ^{3} \\
&&-cu\phi ^{2}+2\left\langle \nabla \left( u\phi ^{2}\right) ,\nabla \phi
^{2}\right\rangle .
\end{eqnarray*}%
Since $\phi $ has support in $D\left( 3R\right) \backslash D\left( \frac{R}{2%
}\right) ,$ we know that $f\geq \frac{1}{2}R$ on the support of $\phi $.
Hence, we may choose $a:=1-\frac{C}{R}$ with a sufficiently large constant $%
C>0$ such that 
\begin{equation*}
2a-\frac{c}{1-a}\frac{S}{f}\geq 1
\end{equation*}%
on the support of $\phi $. As a result, the function $G:=u\phi ^{2}$
verifies 
\begin{equation}
\text{ }\phi ^{2}\Delta _{f}G\geq S^{a-1}G^{2}-cG^{\frac{3}{2}%
}-cG+2\left\langle \nabla G,\nabla \phi ^{2}\right\rangle .  \label{e1}
\end{equation}%
Since $a<1$ and $S^{a-1}\geq A^{a-1},$ the maximum principle implies that

\begin{equation*}
G\leq c,
\end{equation*}%
for some constant $c$ depending on $A$. Hence, on $D\left( 2R\right)
\backslash D\left( R\right) ,$ 
\begin{equation*}
\frac{\left\vert \mathrm{Ric}\right\vert ^{2}}{S}=GS^{a-1}\leq cS^{a-1}.
\end{equation*}%
Let us recall a result in \cite{CLY} that there exists a constant $c>0$ so
that $Sf\geq c$ on $M$. In our context, this constant has the dependency as
stated in the conclusion of the proposition.

Since $a-1=-\frac{C}{R}$ and $S\geq \frac{c}{R}$ on $D\left( 2R\right) $, it
follows that $S^{a-1}\leq c$ on $D\left( 2R\right) \backslash D\left(
R\right) $. Therefore, 
\begin{equation*}
\frac{\left\vert \mathrm{Ric}\right\vert ^{2}}{S}\leq c
\end{equation*}%
on $D\left( 2R\right) \backslash D\left( R\right) $. Since $R$ is arbitrary,
this proves the result.
\end{proof}

Proposition \ref{Rc} implies in particular that the Ricci curvature is
bounded on $M$. We now prove that the curvature of four dimensional
shrinking gradient Ricci solitons with bounded scalar curvature is bounded.

\begin{theorem}
\label{Main1}Let $\left( M,g,f\right) $ be a four dimensional shrinking
Ricci soliton with bounded scalar curvature $S\le A$. Then the Riemann
curvature tensor and its covariant derivative are bounded in norm as well.
More precisely, 
\begin{equation*}
\sup_{M}\left( \left\vert \mathrm{Rm}\right\vert +\left\vert \nabla \mathrm{%
Rm}\right\vert \right) \leq C,
\end{equation*}%
where $C>0$ is a constant depending only on $A$ and $\sup_{D\left(
r_{0}\right) }\left\vert \mathrm{Rm}\right\vert $.
\end{theorem}

\begin{proof}
We first show that 
\begin{equation}
\sup_{M}\left\vert \mathrm{Rm}\right\vert \leq c\text{.}  \label{R-1}
\end{equation}%
Using (\ref{id}) and the Kato inequality, one sees that 
\begin{equation*}
\Delta _{f}\left\vert \mathrm{Rm}\right\vert \geq -c\left\vert \mathrm{Rm}%
\right\vert ^{2}.
\end{equation*}%
Rewrite this into

\begin{equation}
\Delta _{f}\left\vert \mathrm{Rm}\right\vert \geq \left\vert \mathrm{Rm}%
\right\vert ^{2}-\left( c+1\right) \left\vert \mathrm{Rm}\right\vert ^{2}.
\label{R1}
\end{equation}%
By Proposition \ref{Rc} and Proposition \ref{Curv}, we have

\begin{equation*}
\left\vert \mathrm{Rm}\right\vert ^{2}\leq c\left( \frac{1}{f}\left\vert
\nabla \mathrm{Ric}\right\vert ^{2}+1\right) .
\end{equation*}%
Plugging into (\ref{R1}), we conclude

\begin{equation}
\Delta _{f}\left\vert \mathrm{Rm}\right\vert \geq \left\vert \mathrm{Rm}%
\right\vert ^{2}-\frac{c}{f}\left\vert \nabla \mathrm{Ric}\right\vert ^{2}-c.
\label{R2}
\end{equation}%
On the other hand, we know from (\ref{F1}) that 
\begin{equation}
\Delta _{f}\left\vert \mathrm{Ric}\right\vert ^{2}\geq \left\vert \nabla 
\mathrm{Ric}\right\vert ^{2}-c.  \label{R3}
\end{equation}%
Therefore, combining (\ref{R2}) and (\ref{R3}), we obtain

\begin{eqnarray*}
\Delta _{f}\left( \left\vert \mathrm{Rm}\right\vert +\left\vert \mathrm{Ric}%
\right\vert ^{2}\right) &\geq &\left\vert \mathrm{Rm}\right\vert ^{2}-c \\
&\geq &\frac{1}{2}\left( \left\vert \mathrm{Rm}\right\vert +\left\vert 
\mathrm{Ric}\right\vert ^{2}\right) ^{2}-c.
\end{eqnarray*}%
In other words, the function 
\begin{equation*}
v:=\left\vert \mathrm{Rm}\right\vert +\left\vert \mathrm{Ric}\right\vert ^{2}
\end{equation*}%
satisfies the following differential inequality on $M\backslash D\left(
r_{0}\right).$

\begin{equation*}
\Delta _{f}v\geq \frac{1}{2}v^{2}-c.
\end{equation*}%
Arguing as in Proposition \ref{Rc}, we conclude that $v$ is bounded by a
constant on $M\backslash D\left( r_{0}\right) .$ This implies (\ref{R-1}).

Now we use Shi's derivative estimates to prove that 
\begin{equation}
\left\vert \nabla \mathrm{Rm}\right\vert \leq c  \label{T2}
\end{equation}%
for some constant $c>0$. Note that

\begin{eqnarray*}
\Delta _{f}\left\vert \nabla \mathrm{Rm}\right\vert ^{2} &\geq &2\left\vert
\nabla ^{2}\mathrm{Rm}\right\vert ^{2}-c\left\vert \nabla \mathrm{Rm}%
\right\vert ^{2}\left\vert \mathrm{Rm}\right\vert \\
&\geq &2\left\vert \nabla \left\vert \nabla \mathrm{Rm}\right\vert
\right\vert ^{2}-C\left\vert \nabla \mathrm{Rm}\right\vert ^{2}.
\end{eqnarray*}%
This implies

\begin{equation*}
\Delta _{f}\left\vert \nabla \mathrm{Rm}\right\vert \geq -C\left\vert \nabla 
\mathrm{Rm}\right\vert .
\end{equation*}%
We also know from (\ref{id}) that 
\begin{eqnarray*}
\Delta _{f}\left\vert \mathrm{Rm}\right\vert ^{2} &\geq &2\left\vert \nabla 
\mathrm{Rm}\right\vert ^{2}-c\left\vert \mathrm{Rm}\right\vert ^{3} \\
&\geq &2\left\vert \nabla \mathrm{Rm}\right\vert ^{2}-c\text{. }
\end{eqnarray*}%
Hence,

\begin{equation*}
\Delta _{f}\left( \left\vert \nabla \mathrm{Rm}\right\vert +\left\vert 
\mathrm{Rm}\right\vert ^{2}\right) \geq \left( \left\vert \nabla \mathrm{Rm}%
\right\vert +\left\vert \mathrm{Rm}\right\vert ^{2}\right) ^{2}-c\text{. }
\end{equation*}%
Now a maximum principle argument as above shows that $\left\vert \nabla 
\mathrm{Rm}\right\vert +\left\vert \mathrm{Rm}\right\vert ^{2}$ is bounded
on $M.$ So (\ref{T2}) follows. This proves the theorem.
\end{proof}

\section{Improved curvature estimates}

In this section, we aim to prove Theorem \ref{Main}. First, we establish a
result similar to Proposition \ref{Rc} for the full curvature tensor. We
continue to follow the notations in the previous section.

\begin{proposition}
\label{Rm}Let $\left( M,g,f\right) $ be a four dimensional shrinking
gradient Ricci soliton with bounded scalar curvature $S\le A.$ Then

\begin{equation*}
\sup_{M}\frac{\left\vert \mathrm{Rm}\right\vert ^{2}}{S}\leq C
\end{equation*}%
for a constant $C>0$ depending only on $A$ and $\sup_{D\left( r_{0}\right)
}\left\vert \mathrm{Rm}\right\vert $.
\end{proposition}

\begin{proof}
According to Proposition \ref{Curv}, 
\begin{eqnarray}
\left\vert \mathrm{Rm}\right\vert ^{2} &\leq &c\left( \frac{\left\vert
\nabla \mathrm{Ric}\right\vert ^{2}}{f}+\frac{c}{f^{2}}+\left\vert \mathrm{%
Ric}\right\vert ^{2}\right)  \label{f1} \\
&\leq &c\left( \frac{1}{f}+S\right)  \notag \\
&\leq & c\,S.  \notag
\end{eqnarray}%
In the second and third line above we have used Proposition \ref{Rc},
Theorem \ref{Main1} and the fact that $\frac{1}{f}\leq cS$ from \cite{CLY},
respectively. This proves the proposition.
\end{proof}

We continue with a similar estimate for the covariant derivative of
curvature.

\begin{proposition}
\label{nabla_Rm}Let $\left( M,g,f\right) $ be a four dimensional shrinking
gradient Ricci soliton with bounded scalar curvature. Then there exists a
constant $C>0$ so that%
\begin{equation*}
\sup_{M}\frac{\left\vert \nabla \mathrm{Rm}\right\vert ^{2}}{S}\leq C.
\end{equation*}
\end{proposition}

\begin{proof}
Let us first prove the following inequality,%
\begin{equation}
\Delta _{f}\left\vert \nabla \mathrm{Rm}\right\vert ^{2}\geq 2\left\vert
\nabla ^{2}\mathrm{Rm}\right\vert ^{2}+3\left\vert \nabla \mathrm{Rm}%
\right\vert ^{2}-c\left\vert \mathrm{Rm}\right\vert \left\vert \nabla 
\mathrm{Rm}\right\vert ^{2}.  \label{u1}
\end{equation}%
We have that 
\begin{equation*}
\Delta _{f}\left\vert \nabla \mathrm{Rm}\right\vert ^{2}=2\left\vert \nabla
^{2}\mathrm{Rm}\right\vert ^{2}+2\left\langle \Delta _{f}\left( \nabla 
\mathrm{Rm}\right) ,\nabla \mathrm{Rm}\right\rangle .
\end{equation*}%
Now we compute 
\begin{eqnarray*}
\Delta _{f}\left( \nabla _{q}\mathrm{Rm}\right) &=&\nabla _{p}\nabla
_{p}\nabla _{q}R_{ijkl}-\nabla _{p}\left( \nabla _{q}R_{ijkl}\right) f_{p} \\
&=&\nabla _{p}\nabla _{q}\nabla _{p}R_{ijkl}-\nabla _{q}\left( \nabla
_{p}R_{ijkl}\right) f_{p}+\mathrm{Rm}\ast \nabla \mathrm{Rm} \\
&=&\nabla _{q}\nabla _{p}\nabla _{p}R_{ijkl}-\nabla _{q}\left( \nabla
_{p}R_{ijkl}f_{p}\right) +f_{pq}\left( \nabla _{p}R_{ijkl}\right) +\mathrm{Rm%
}\ast \nabla \mathrm{Rm} \\
&=&\nabla _{q}\left( \Delta _{f}R_{ijkl}\right) +\frac{1}{2}\nabla
_{q}R_{ijkl}+\mathrm{Rm}\ast \nabla \mathrm{Rm} \\
&=&\frac{3}{2}\nabla _{q}R_{ijkl}+\mathrm{Rm}\ast \nabla \mathrm{Rm}.
\end{eqnarray*}%
In these equalities we used the Ricci identities and the formulas $%
R_{ijkl}f_{l}=\nabla_{j} R_{ik}-\nabla_{i} R_{jk}$ and $\Delta _{f}\mathrm{Rm%
}=\mathrm{Rm}+\mathrm{Rm}\ast \mathrm{Rm}$ from (\ref{id}). Hence, (\ref{u1}%
) is proved.

Using (\ref{u1}) we get%
\begin{align}
\Delta _{f}\left( \left\vert \nabla \mathrm{Rm}\right\vert ^{2}S^{-1}\right)
& \geq S^{-1}\left( 2\left\vert \nabla ^{2}\mathrm{Rm}\right\vert
^{2}+3\left\vert \nabla \mathrm{Rm}\right\vert ^{2}-c\left\vert \mathrm{Rm}%
\right\vert \left\vert \nabla \mathrm{Rm}\right\vert ^{2}\right)  \label{A1}
\\
& +\left\vert \nabla \mathrm{Rm}\right\vert ^{2}\left( -S^{-1}+2\left\vert 
\mathrm{Ric}\right\vert ^{2}S^{-2}+2\left\vert \nabla S\right\vert
^{2}S^{-3}\right)  \notag \\
& -4\left\vert \nabla S\right\vert \left\vert \nabla ^{2}\mathrm{Rm}%
\right\vert S^{-2}\left\vert \nabla \mathrm{Rm}\right\vert  \notag \\
& \geq 2\left\vert \nabla \mathrm{Rm}\right\vert ^{2}S^{-1}-c\left\vert 
\mathrm{Rm}\right\vert \left\vert \nabla \mathrm{Rm}\right\vert ^{2}S^{-1}. 
\notag
\end{align}%
To derive the last line of (\ref{A1}) we have used that 
\begin{eqnarray*}
2\left\vert \left\langle \nabla \left\vert \nabla \mathrm{Rm}\right\vert
^{2},\nabla S^{-1}\right\rangle \right\vert &\leq &4\left\vert \nabla ^{2}%
\mathrm{Rm}\right\vert \left\vert \nabla \mathrm{Rm}\right\vert
S^{-2}\left\vert \nabla S\right\vert \\
&\leq &2\left\vert \nabla ^{2}\mathrm{Rm}\right\vert ^{2}S^{-1}+2\left\vert
\nabla S\right\vert ^{2}S^{-3}\left\vert \nabla \mathrm{Rm}\right\vert ^{2}.
\end{eqnarray*}%
Using Proposition \ref{Rm} and (\ref{T2}) we can bound 
\begin{eqnarray*}
c\left\vert \mathrm{Rm}\right\vert \left\vert \nabla \mathrm{Rm}\right\vert
^{2}S^{-1} &\leq &c\left\vert \nabla \mathrm{Rm}\right\vert S^{-\frac{1}{2}}
\\
&\leq &\left\vert \nabla \mathrm{Rm}\right\vert ^{2}S^{-1}+c.
\end{eqnarray*}%
Therefore, the function 
\begin{equation*}
w:=\left\vert \nabla \mathrm{Rm}\right\vert ^{2}S^{-1}-c
\end{equation*}%
satisfies 
\begin{equation}
\Delta _{f}w\geq w.  \label{A5}
\end{equation}%
Our goal is to show that $w$ must be bounded above. We use the maximum
principle again.

Let $\psi \left( t\right) =\frac{R-t}{R}$ on $\left[ 0,R\right]$ and $\psi
=0 $ for $t\geq R$. Then $\psi \left( f\right) $ as a cutoff function on $M$
satisfies 
\begin{eqnarray}
\left\vert \nabla \psi \right\vert &=&\frac{\left\vert \nabla f\right\vert }{%
R}  \label{A5'} \\
\Delta _{f}\psi &=&\frac{1}{R}\left( f-2\right) .  \notag
\end{eqnarray}%
Therefore, for $G:=\psi ^{2}\,w,$ using (\ref{A5}), we have

\begin{equation}
\Delta _{f}G\geq \left( 1+\psi ^{-1}\frac{2}{R}\left( f-2\right) -6\psi
^{-2}\left\vert \nabla \psi \right\vert ^{2}\right) G+4\psi
^{-1}\left\langle \nabla G,\nabla \psi \right\rangle .  \label{A6}
\end{equation}%
Suppose that $G\left( q\right) >0$ at the maximum point $q$ of $G.$ Then (%
\ref{A6}) implies that 
\begin{equation}
\frac{2}{R}\left( f-2\right) \psi \leq 6\left\vert \nabla \psi \right\vert
^{2}\leq 6\frac{1}{R^{2}}f.  \label{A7}
\end{equation}

If $q\in D\left( r_{0}\right) $, then 
\begin{eqnarray*}
\sup_{D\left( \frac{R}{2}\right) }\left( \left\vert \nabla \mathrm{Rm}%
\right\vert ^{2}S^{-1}\right) &\leq & c+4\,\sup_{D\left( \frac{R}{2}\right)
}G \\
&\leq & c+4\,\sup_{D\left( r_{0}\right) }G \\
&\leq &c.
\end{eqnarray*}%
On the other hand, if $q\in M\backslash D\left( r_{0}\right) ,$ then $%
f\left( q\right) -2\geq \frac{1}{2}f\left( q\right).$ By (\ref{A7}), $\psi
\left( q\right) R\leq 6.$ This shows that $f\left( q\right) \geq R-6.$
Therefore,

\begin{eqnarray*}
\frac{1}{4}\sup_{D\left( \frac{R}{2}\right) }\left( \left\vert \nabla 
\mathrm{Rm}\right\vert ^{2}S^{-1}-c\right) &\leq &\sup_{D\left( \frac{R}{2}%
\right) }G \\
&\leq &G\left( q\right) \\
&\leq &\frac{36}{R^{2}}\sup_{D\left( R\right) }\left( \left\vert \nabla 
\mathrm{Rm}\right\vert ^{2}S^{-1}\right) \\
&\leq &\frac{c}{R},
\end{eqnarray*}%
where in the last line we have used (\ref{T2}) and that $Sf\geq c>0$ by \cite%
{CLY}.

In conclusion, we have proved that if $G\left( q\right) >0$, then 
\begin{equation*}
\sup_{D\left( \frac{R}{2}\right) }\left( \left\vert \nabla \mathrm{Rm}%
\right\vert ^{2}S^{-1}\right) \leq c.
\end{equation*}%
On the other hand, if at the maximum point $q$ of $G$ we have $G\left(
q\right) \leq 0,$ then $w$ is nonpositive on $D\left( R\right),$ which again
implies

\begin{equation*}
\sup_{D\left( \frac{R}{2}\right) }\left( \left\vert \nabla \mathrm{Rm}%
\right\vert ^{2}S^{-1}\right) \leq c.
\end{equation*}%
This proves the proposition.
\end{proof}

We now wish to establish a gradient estimate for the scalar curvature. This
will be improved later.

\begin{lemma}
\label{Scalar_first}Let $\left( M,g,f\right) $ be a four dimensional
shrinking gradient Ricci soliton with bounded scalar curvature $S\leq A.$
Then there exists a constant $C>0$ so that 
\begin{equation*}
\left\vert \nabla \ln S\right\vert ^{2}\leq C\,\ln (f+2)\text{ on }M.
\end{equation*}
\end{lemma}

\begin{proof}
We adopt an argument in \cite{MW}. Let $h:=\frac{1}{\epsilon }S^{\epsilon }$
with $\epsilon >0$ small to be determined later. Then a direct computation
gives 
\begin{equation*}
\Delta _{f}h=\epsilon h-2\epsilon \left\vert \mathrm{Ric}\right\vert
^{2}S^{-1}h+\left( \epsilon -1\right) S^{\epsilon -2}\left\vert \nabla
S\right\vert ^{2}.
\end{equation*}

Let us denote $\sigma :=\left\vert \nabla h\right\vert ^{2}=S^{2\epsilon
-2}\left\vert \nabla S\right\vert ^{2}$. The Bochner formula asserts that 
\begin{align*}
\frac{1}{2}\Delta _{f}\sigma & =\left\vert \mathrm{Hess}\left( h\right)
\right\vert ^{2}+\left\langle \nabla h,\nabla \left( \Delta _{f}h\right)
\right\rangle +\mathrm{Ric}_{f}\left( \nabla h,\nabla h\right) \\
& \geq \left\langle \nabla h,\nabla \left( \Delta _{f}h\right) \right\rangle
\\
& \geq \left( \epsilon -1\right) \left\langle \nabla h,\nabla \left(
S^{\epsilon -2}\left\vert \nabla S\right\vert ^{2}\right) \right\rangle
-2\epsilon \left\langle \nabla h,\nabla \left( \left\vert \mathrm{Ric}%
\right\vert ^{2}S^{-1}h\right) \right\rangle .
\end{align*}%
Note that 
\begin{eqnarray*}
\left\langle \nabla h,\nabla \left( S^{\epsilon -2}\left\vert \nabla
S\right\vert ^{2}\right) \right\rangle &=&\left\langle \nabla h,\nabla
\left( S^{-\epsilon }\sigma \right) \right\rangle \\
&=&-\epsilon \left\langle \nabla h,\nabla S\right\rangle S^{-\epsilon
-1}\sigma +S^{-\epsilon }\left\langle \nabla h,\nabla \sigma \right\rangle \\
&=&-\epsilon \left\vert \nabla h\right\vert ^{2}S^{-2\epsilon }\sigma
+S^{-\epsilon }\left\langle \nabla h,\nabla \sigma \right\rangle .
\end{eqnarray*}%
Furthermore, we have 
\begin{eqnarray*}
-\epsilon \left\langle \nabla h,\nabla \left( \left\vert \mathrm{Ric}%
\right\vert ^{2}S^{-1}h\right) \right\rangle &\geq &-2\epsilon \left\vert
\nabla \mathrm{Ric}\right\vert \left\vert \mathrm{Ric}\right\vert
h\left\vert \nabla h\right\vert S^{-1}-\epsilon \left\vert \mathrm{Ric}%
\right\vert ^{2}S^{-1}\left\vert \nabla h\right\vert ^{2} \\
&\geq &-c\epsilon h\left\vert \nabla h\right\vert -c \epsilon \sigma \\
&\geq &-c- c\sigma ,
\end{eqnarray*}%
where in the second line we used Proposition \ref{Rc} and Proposition \ref%
{nabla_Rm} to bound  $\left\vert \nabla \mathrm{Ric}\right\vert \left\vert 
\mathrm{Ric}\right\vert S^{-1}\leq c,$ and in the last line we used $%
\epsilon h=S^{\epsilon }\leq c$.

Consequently, 
\begin{equation}
\frac{1}{2}\Delta _{f}\sigma \geq \epsilon \left( 1-\epsilon \right)
S^{-2\epsilon }\sigma ^{2}+\left( \epsilon -1\right) S^{-\epsilon
}\left\langle \nabla h,\nabla \sigma \right\rangle -c \sigma -c.  \label{z1}
\end{equation}%
Let $\phi $ be a smooth non-negative function defined on the real line so
that $\phi \left( t\right) =1$ for $0\leq t\leq R$ and $\phi \left( t\right)
=0$ for $t\geq 2R$. We may choose $\phi $ so that 
\begin{equation*}
t^{2}\left( \left\vert \phi ^{\prime }\right\vert ^{2}\left( t\right)
+\left\vert \phi ^{\prime \prime }\right\vert \left( t\right) \right) \leq c.
\end{equation*}
We use $\phi \left( f\left( x\right) \right) $ as a cut-off function with
support in $D\left( 2R\right) .$ Note that we have $\left\vert \nabla \phi
\right\vert \leq \frac{c}{\sqrt{R}}$ and $\left\vert \Delta _{f}\phi
\right\vert \leq c$ for a universal constant $c>0.$

Let $G:=\phi ^{2}\sigma $. From (\ref{z1}), we find that%
\begin{align*}
\frac{1}{2}\phi ^{2}\Delta _{f}G& =\frac{1}{2}\phi ^{4}\left( \Delta
_{f}\sigma \right) +\frac{1}{2}G\left( \Delta _{f}\phi ^{2}\right) +\phi
^{2}\left\langle \nabla \phi ^{2},\nabla \sigma \right\rangle \\
& \geq \epsilon \left( 1-\epsilon \right) S^{-2\epsilon }\ G^{2}+\left(
\epsilon -1\right) S^{-\epsilon }\left\langle \nabla h,\nabla G\right\rangle
\phi ^{2}-\left( \epsilon -1\right) S^{-\epsilon }\left\langle \nabla
h,\nabla \phi ^{2}\right\rangle G \\
& -cG-c+\ \left\langle \nabla \phi ^{2},\nabla G\right\rangle .
\end{align*}%
At the maximum point of $G$ we have%
\begin{eqnarray}
\epsilon G^{2} &\leq &c\ G^{\frac{3}{2}}\left\vert \nabla \phi \right\vert
S^{\epsilon }+cG+c  \label{z2} \\
&\leq &\frac{c}{\sqrt{R}}G^{\frac{3}{2}}+cG+c.  \notag
\end{eqnarray}

We now choose $\epsilon :=\left( \ln R\right) ^{-1}$. It is easy to see that
(\ref{z2}) implies%
\begin{equation*}
\sup_{M}G\leq \frac{c}{\epsilon }=c\,\ln R.
\end{equation*}%
This proves that 
\begin{equation*}
\sup_{D\left( R\right) }\left( S^{2\epsilon }\left\vert \nabla \ln
S\right\vert ^{2}\right) \leq c\, \ln R.
\end{equation*}%
Using the bound in \cite{CLY} that $S\geq \frac{c}{R}$ on $D\left( R\right) $%
, one easily concludes that $S^{2\epsilon }\ge c>0.$ Thus, 
\begin{equation*}
\sup_{\Sigma\left( R\right) }\left\vert \nabla \ln S\right\vert ^{2}\leq
c\,\ln (f+2)
\end{equation*}
and the result follows.
\end{proof}

To prove Theorem \ref{Main} in the introduction, we need to improve the
Ricci curvature estimate from\ Proposition \ref{Rc}. Let us first establish
a parallel version of Lemma \ref{DI}.

\begin{lemma}
\label{DI_improved}Let $\left( M,g,f\right) $ be a four dimensional
shrinking gradient Ricci soliton with bounded scalar curvature $S\le A.$
Then the function $u:=\left\vert \mathrm{Ric}\right\vert ^{2}S^{-2}$
verifies the differential inequality 
\begin{equation*}
\ \ \Delta_{F}u\geq 3\,u^{2}\,S-c\,u\,S
\end{equation*}%
on $M\backslash D\left( r_{0}\right) $ for some $r_{0}>0$ which depends only
on $A$. Here, $c>0$ is a universal constant and 
\begin{equation*}
F:=f-2\ln S.
\end{equation*}
\end{lemma}

\begin{proof}
Note that by Proposition \ref{Curv} \ and Proposition \ref{nabla_Rm} we have 
\begin{eqnarray*}
\left\vert \mathrm{Rm}\right\vert &\leq &c\left( \frac{\left\vert \nabla 
\mathrm{Ric}\right\vert }{\sqrt{f}}+\frac{\left\vert \mathrm{Ric}\right\vert
^{2}+1}{f}+\left\vert \mathrm{Ric}\right\vert \right) \\
&\leq &c\left( \frac{\sqrt{S}}{\sqrt{f}}+\frac{1}{f}+\left\vert \mathrm{Ric}%
\right\vert \right) \\
&\leq &c\left( S+\left\vert \mathrm{Ric}\right\vert \right) .
\end{eqnarray*}%
In the last line we have used the fact that $S\geq \frac{c}{f}$ on $M$.
Since $S\leq 2\left\vert \mathrm{Ric}\right\vert ,$ we conclude that 
\begin{equation}
\left\vert \mathrm{Rm}\right\vert \leq c\left\vert \mathrm{Ric}\right\vert .
\label{Rm_bd}
\end{equation}

By (\ref{id}) we have 
\begin{eqnarray*}
\Delta _{f}\left\vert \mathrm{Ric}\right\vert ^{2} &\geq &2\left\vert \nabla 
\mathrm{Ric}\right\vert ^{2}+2\left\vert \mathrm{Ric}\right\vert
^{2}-c\left\vert \mathrm{Rm}\right\vert \left\vert \mathrm{Ric}\right\vert
^{2} \\
&\geq &2\left\vert \nabla \mathrm{Ric}\right\vert ^{2}+2\left\vert \mathrm{%
Ric}\right\vert ^{2}-c\left\vert \mathrm{Ric}\right\vert ^{3}.
\end{eqnarray*}%
Hence,

\begin{eqnarray}
&&\Delta _{f}\left( \left\vert \mathrm{Ric}\right\vert ^{2}S^{-2}\right)
\label{AA1} \\
&=&S^{-2}\Delta _{f}\left( \left\vert \mathrm{Ric}\right\vert ^{2}\right)
+\left\vert \mathrm{Ric}\right\vert ^{2}\Delta _{f}\left( S^{-2}\right)
+2\left\langle \nabla S^{-2},\nabla \left\vert \mathrm{Ric}\right\vert
^{2}\right\rangle  \notag \\
&\geq& 2\left\vert \nabla \mathrm{Ric}\right\vert ^{2}S^{-2}+2\left\vert 
\mathrm{Ric}\right\vert ^{2}S^{-2}-c\left\vert \mathrm{Ric}\right\vert
^{3}S^{-2} +2\left\langle \nabla S^{-2},\nabla \left\vert \mathrm{Ric}%
\right\vert ^{2}\right\rangle  \notag \\
&+&\left\vert \mathrm{Ric}\right\vert ^{2}\left( -2S^{-2}+4\left\vert 
\mathrm{Ric}\right\vert ^{2}S^{-3}+6\left\vert \nabla S\right\vert
^{2}S^{-4}\right).  \notag
\end{eqnarray}%
We can estimate 
\begin{eqnarray*}
2\left\langle \nabla S^{-2},\nabla \left\vert \mathrm{Ric}\right\vert
^{2}\right\rangle &=&S^{2}\left\langle \nabla S^{-2},\nabla \left(
\left\vert \mathrm{Ric}\right\vert ^{2}S^{-2}\right) \right\rangle \\
&&-\left\langle \nabla S^{-2},\nabla S^{-2}\right\rangle \left\vert \mathrm{%
Ric}\right\vert ^{2}S^{2}+\left\langle \nabla S^{-2},\nabla \left\vert 
\mathrm{Ric}\right\vert ^{2}\right\rangle \\
&\geq &-2\left\langle \nabla \ln S,\nabla \left( \left\vert \mathrm{Ric}%
\right\vert ^{2}S^{-2}\right) \right\rangle -4\left\vert \nabla S\right\vert
^{2}S^{-4}\left\vert \mathrm{Ric}\right\vert ^{2} \\
&&-4\left\vert \nabla \mathrm{Ric}\right\vert \left\vert \nabla S\right\vert
S^{-3}\left\vert \mathrm{Ric}\right\vert \\
&\geq &-2\left\langle \nabla \ln S,\nabla \left( \left\vert \mathrm{Ric}%
\right\vert ^{2}S^{-2}\right) \right\rangle -6\left\vert \nabla S\right\vert
^{2}S^{-4}\left\vert \mathrm{Ric}\right\vert ^{2} \\
&&-2\left\vert \nabla \mathrm{Ric}\right\vert ^{2}S^{-2}.
\end{eqnarray*}%
Plugging this in (\ref{AA1}) we get 
\begin{eqnarray}
\Delta _{F}\left( \left\vert \mathrm{Ric}\right\vert ^{2}S^{-2}\right) &\geq
&\ 4\left\vert \mathrm{Ric}\right\vert ^{4}S^{-3}-c\left\vert \mathrm{Ric}%
\right\vert ^{3}S^{-2}  \label{AA2} \\
&\geq &3\left\vert \mathrm{Ric}\right\vert ^{4}S^{-3}-c\left\vert \mathrm{Ric%
}\right\vert ^{2}S^{-1},  \notag
\end{eqnarray}
where we have used the Cauchy-Schwarz inequality 
\begin{equation*}
c\left\vert \mathrm{Ric}\right\vert ^{3}S^{-2}\leq \left\vert \mathrm{Ric}%
\right\vert ^{4}S^{-3}+c\left\vert \mathrm{Ric}\right\vert ^{2}S^{-1}
\end{equation*}%
in the last line. This proves the result.
\end{proof}

We are ready to prove the following result which was stated as Theorem \ref%
{Main} in the introduction.

\begin{theorem}
\label{Rc_improved}Let $\left( M,g,f\right) $ be a four dimensional
shrinking gradient Ricci soliton with bounded scalar curvature $S\le A.$
Then there exists a constant $C>0,$ depending only on $A$ and $\sup_{D\left(
r_{0}\right) }\left\vert \mathrm{Rm}\right\vert ,$ so that%
\begin{equation}
\sup_{M}\frac{\left\vert \mathrm{Rm}\right\vert }{S}\leq C.  \label{R0}
\end{equation}
\end{theorem}

\begin{proof}
By (\ref{Rm_bd}) it suffices to show that 
\begin{equation}
\sup_{M}\frac{\left\vert \mathrm{Ric}\right\vert }{S}\leq C.  \label{B0}
\end{equation}%
By Lemma \ref{DI_improved} the function $u:=\left\vert \mathrm{Ric}%
\right\vert ^{2}S^{-2}$ verifies the following differential inequality.

\begin{equation}
\Delta _{F}u\geq 3\,u^{2}\,S-c\,u\,S  \label{B1}
\end{equation}%
on $M\backslash D\left( r_{0}\right) .$

Let $\psi \left( t\right) =\frac{R-t}{R}$ on $\left[ 0,R\right] $ and $\psi
=0$ for $t\geq R$. Using $\psi \left( f\right) $ as a cutoff function on $M,$
we have

\begin{eqnarray}
\left\vert \nabla \psi \right\vert &=&\frac{1}{R}\left\vert \nabla
f\right\vert  \label{B1'} \\
\Delta _{f}\psi &=&\frac{1}{R}\left( f-2\right) .  \notag
\end{eqnarray}%
By Lemma \ref{Scalar_first}, 
\begin{eqnarray*}
\Delta _{F}\psi &=&\Delta _{f}\psi +2\left\langle \nabla \ln S,\nabla \psi
\right\rangle \\
&\geq &\frac{1}{R}\left( f-2\right) -\frac{2}{R}\left\vert \nabla \ln
S\right\vert \left\vert \nabla f\right\vert \\
&\geq &\frac{1}{R}\left( f-2\right) -\frac{c}{R}\sqrt{f}\ln \left(
f+2\right) .
\end{eqnarray*}%
This shows that there exists a constant $r_{0}>0$ so that on $D\left(
R\right) \backslash D\left( r_{0}\right) ,$%
\begin{equation}
\Delta _{F}\psi \geq 0.  \label{B2}
\end{equation}%
Using (\ref{B1}) and (\ref{B2}), for the function $G:=\psi ^{2}u$ we have
that on $M\backslash D\left( r_{0}\right) ,$%
\begin{eqnarray}
\psi ^{2}\Delta _{F}G &\geq &3G^{2}S-cGS+G\Delta _{F}\psi ^{2}+2\left\langle
\nabla u,\nabla \psi ^{2}\right\rangle \psi ^{2}  \label{B3} \\
&\geq &3G^{2}S-cGS+2\left\langle \nabla \left( G\psi ^{-2}\right) ,\nabla
\psi ^{2}\right\rangle \psi ^{2}  \notag \\
&=&3G^{2}S-cGS+2\left\langle \nabla G,\nabla \psi ^{2}\right\rangle
-8\left\vert \nabla \psi \right\vert ^{2}G.  \notag
\end{eqnarray}%
By (\ref{B1'}) and the estimate $Sf\geq c>0$ we have 
\begin{eqnarray*}
\left\vert \nabla \psi \right\vert ^{2}G &\leq &\frac{1}{R}G \\
&\leq &\frac{1}{c}SG.
\end{eqnarray*}%
Therefore, (\ref{B3}) becomes 
\begin{equation*}
\psi ^{2}\Delta _{F}G\geq \left( 3G^{2}-cG\right) S+2\left\langle \nabla
G,\nabla \psi ^{2}\right\rangle .
\end{equation*}%
Now the maximum principle implies that $G$ must be bounded. This proves (\ref%
{B0}) and hence the theorem.
\end{proof}

We can now improve the covariant derivative estimate in Proposition \ref%
{nabla_Rm} as well.

\begin{theorem}
\label{Scalar}Let $\left( M,g,f\right) $ be a four dimensional shrinking
gradient Ricci soliton with bounded scalar curvature $S\leq A.$ Then 
\begin{equation*}
\left\vert \nabla \mathrm{Rm}\right\vert \leq C\,S\text{ \ on }M
\end{equation*}%
for a constant $C>0$ depending only on $A$ and $\sup_{D\left( r_{0}\right)
}\left\vert \mathrm{Rm}\right\vert $. In particular, 
\begin{equation*}
\sup_{M}\,\left\vert \nabla \ln S\right\vert \leq C.
\end{equation*}
\end{theorem}

\begin{proof}
Using (\ref{u1}) we get 
\begin{align}
\Delta _{f}\left( \left\vert \nabla \mathrm{Rm}\right\vert ^{2}S^{-2}\right)
& \geq S^{-2}\left( 2\left\vert \nabla ^{2}\mathrm{Rm}\right\vert
^{2}+3\left\vert \nabla \mathrm{Rm}\right\vert ^{2}-c\left\vert \mathrm{Rm}%
\right\vert \left\vert \nabla \mathrm{Rm}\right\vert ^{2}\right)  \label{u2}
\\
& +\left\vert \nabla \mathrm{Rm}\right\vert ^{2}\left( -2S^{-2}+4\left\vert 
\mathrm{Ric}\right\vert ^{2}S^{-3}+6\left\vert \nabla S\right\vert
^{2}S^{-4}\right)  \notag \\
& +2\left\langle \nabla \left\vert \nabla \mathrm{Rm}\right\vert ^{2},\nabla
S^{-2}\right\rangle .  \notag
\end{align}%
Observe that

\begin{eqnarray*}
2\left\langle \nabla \left\vert \nabla \mathrm{Rm}\right\vert ^{2},\nabla
S^{-2}\right\rangle &=&\left\langle \nabla \left( \left\vert \nabla \mathrm{%
Rm}\right\vert ^{2}S^{-2}\right) S^{2},\nabla S^{-2}\right\rangle
+\left\langle \nabla \left\vert \nabla \mathrm{Rm}\right\vert ^{2},\nabla
S^{-2}\right\rangle \\
&\geq &\left\langle \nabla \left( \left\vert \nabla \mathrm{Rm}\right\vert
^{2}S^{-2}\right) ,\nabla S^{-2}\right\rangle S^{2}+\left\vert \nabla 
\mathrm{Rm}\right\vert ^{2}S^{-2}\left\langle \nabla S^{2},\nabla
S^{-2}\right\rangle \\
&&-4\left\vert \nabla ^{2}\mathrm{Rm}\right\vert \left\vert \nabla
S\right\vert \left\vert \nabla \mathrm{Rm}\right\vert S^{-3} \\
&\geq &-2\left\langle \nabla \left( \left\vert \nabla \mathrm{Rm}\right\vert
^{2}S^{-2}\right) ,\nabla \ln S\right\rangle -6\left\vert \nabla \mathrm{Rm}%
\right\vert ^{2}\left\vert \nabla S\right\vert ^{2}S^{-4} \\
&&-2\left\vert \nabla ^{2}\mathrm{Rm}\right\vert ^{2}S^{-2}.
\end{eqnarray*}%
It now follows from (\ref{u2})\ and Theorem \ref{Rc_improved} that the
function $w:=\left\vert \nabla \mathrm{Rm}\right\vert ^{2}S^{-2}$ verifies
the inequality 
\begin{eqnarray}
\Delta _{F}w &\geq &w-c\left\vert \mathrm{Rm}\right\vert w  \label{u3} \\
&\geq &w\left( 1-cS\right) ,  \notag
\end{eqnarray}
where $F:=f-2\ln S$. We now show that a function $w\geq 0$ satisfying (\ref%
{u3}) must be bounded. Let $\psi \left( t\right) =\frac{R-t}{R}$ on $\left[
0,R\right] $ and $\psi =0$ for $t\geq R$. We view $\psi \left( f\right) $ as
a cut-off function on $D\left( R\right) .$

For $G:=\psi ^{2}w$ we have that 
\begin{equation}
\Delta _{F}G\geq G\left( 1-cS\right) +2\psi ^{-1}\left( \Delta _{F}\psi
\right) G-6\psi ^{-2}\left\vert \nabla \psi \right\vert ^{2}G+2\psi
^{-2}\left\langle \nabla G,\nabla \psi ^{2}\right\rangle .  \label{u4}
\end{equation}%
Let $q\in D\left( R\right) $ be the maximum point of $G$. If $q\in D\left(
r_{0}\right) ,$ then it follows immediately that $w$ is bounded on $D\left( 
\frac{R}{2}\right) .$ So without loss of generality we may assume $q\in
D\left( R\right) \backslash D\left( r_{0}\right)$. Furthermore, if $S\left(
q\right) >\frac{1}{c},$ where $c>0$ is the constant in (\ref{u4}), then from
the definition of $w$ and Theorem \ref{Main1} one sees that $G\leq G(q) \leq
C$ on $D\left( R\right) $. Again, this proves that $w$ is bounded on $%
D\left( \frac{R}{2}\right) $. So we may assume in (\ref{u4}) that $%
1-cS\left( q\right) \geq 0$. Now the maximum principle implies that at $q$
we have%
\begin{equation}
0\geq \psi ^{-1}\left( \Delta _{F}\psi \right) -3\psi ^{-2}\left\vert \nabla
\psi \right\vert ^{2}.  \label{u5}
\end{equation}

Since $q\in D\left( R\right) \backslash D\left( r_{0}\right) $, we can
estimate 
\begin{eqnarray*}
\Delta _{F}\psi &=&-\frac{1}{R}\Delta _{f}\left( f\right) +\frac{2}{R}%
\left\langle \nabla \ln S,\nabla f\right\rangle \\
&\geq &\frac{f-2}{R}-\frac{2}{R}\left\vert \nabla \ln S\right\vert \sqrt{f}
\\
&\geq &\frac{f-c\,\sqrt{f}\,\ln (f+2)-2}{R} \\
&\geq &\frac{f}{2R},
\end{eqnarray*}%
where in the third line we have used Lemma \ref{Scalar_first}. Therefore, (%
\ref{u5}) implies that at $q,$ 
\begin{equation*}
\frac{f}{R}\psi \leq 6\left\vert \nabla \psi \right\vert ^{2}\leq 6\frac{1}{%
R^{2}}f.
\end{equation*}%
This means $f\left( q\right) \geq R-6$ and

\begin{eqnarray*}
\frac{1}{4}\sup_{D\left( \frac{R}{2}\right) }\left( \left\vert \nabla 
\mathrm{Rm}\right\vert ^{2}S^{-2}\right) &\leq &\sup_{D\left( \frac{R}{2}%
\right) }G \\
&\leq &G\left( q\right) \\
&\leq &\frac{36}{R^{2}}\sup_{D\left( R\right) }\left( \left\vert \nabla 
\mathrm{Rm}\right\vert ^{2}S^{-2}\right) \\
&\leq &\frac{c}{R},
\end{eqnarray*}%
where in the last line we have used Proposition \ref{nabla_Rm} and $Sf\geq
c>0$. This again proves that $\left\vert \nabla \mathrm{Rm}\right\vert
^{2}S^{-2}$ is bounded. In conclusion, we have proved that%
\begin{equation*}
\sup_{D\left( \frac{R}{2}\right) }\left( \left\vert \nabla \mathrm{Rm}%
\right\vert S^{-1}\right) \leq c.
\end{equation*}%
Since $R$ is arbitrary, this proves the theorem.
\end{proof}

\section{Curvature lower bound}

In this section we prove Theorem \ref{Lower_Bound}. The argument uses the
estimates from the previous sections and ideas of Hamilton-Ivey pinching
estimate for three dimensional Ricci flows.

\begin{theorem}
Let $\left( M,g,f\right) $ be a four dimensional shrinking gradient Ricci
soliton with bounded scalar curvature. Then the curvature operator is
bounded below by 
\begin{equation}
\mathrm{Rm}\geq -\left( \frac{c}{\ln f}\right) ^{\frac{1}{4}}.  \label{bound}
\end{equation}
\end{theorem}

\begin{proof}
Note that Theorem \ref{Scalar} implies 
\begin{equation}
\left\vert R_{ijk4}\right\vert \leq c\frac{\left\vert \nabla \mathrm{Ric}%
\right\vert }{\sqrt{f}}\leq cSf^{-\frac{1}{2}}.  \label{t-1}
\end{equation}%
Hence, to establish (\ref{bound}) for the curvature operator of $M,$ it is
enough to do so for its restriction to the subspace $\wedge ^{2}\left(
T\Sigma \right) $. By (\ref{t-1}) and Proposition \ref{Curv}, it is in turn
enough to establish (\ref{bound}) for the curvature operator of $\Sigma $.
To achieve this, we diagonalize $\mathrm{Rm}^{\Sigma }$ on $\wedge
^{2}\left( T\Sigma \right) .$ Since $\Sigma $ is three dimensional, it is
possible to choose an orthonormal frame $\left\{ e_{1},e_{2},e_{3}\right\} $
of $T\Sigma $ such that $e_{a}\wedge e_{b}$ are eigenvectors of the
curvature operator $\mathrm{Rm}^{\Sigma }$. Let $\lambda _{1}^{\Sigma }\leq
\lambda _{2}^{\Sigma }\leq \lambda _{3}^{\Sigma }$ be the eigenvalues of $%
\mathrm{Ric}^{\Sigma }.$ Then 
\begin{eqnarray*}
\nu ^{\Sigma } &=&\lambda _{1}^{\Sigma }+\lambda _{2}^{\Sigma }-\lambda
_{3}^{\Sigma } \\
\lambda ^{\Sigma } &=&\lambda _{1}^{\Sigma }+\lambda _{3}^{\Sigma }-\lambda
_{2}^{\Sigma } \\
\mu ^{\Sigma } &=&\lambda _{2}^{\Sigma }+\lambda _{3}^{\Sigma }-\lambda
_{1}^{\Sigma }
\end{eqnarray*}%
are the eigenvalues of $\mathrm{Rm}^{\Sigma }$ and $\nu ^{\Sigma }\leq
\lambda ^{\Sigma }\leq \mu ^{\Sigma }$. Our goal is to show that 
\begin{equation}
\nu ^{\Sigma }\geq -\left( \frac{c}{\ln f}\right) ^{\frac{1}{4}}.  \label{t0}
\end{equation}%
Restrict the Ricci curvature of $M$ to $\Sigma $ and let $\lambda _{1}\leq
\lambda _{2}\leq \lambda _{3}$ be the eigenvalues of the resulting operator.
Since 
\begin{eqnarray*}
\nu ^{\Sigma } &=&S^{\Sigma }-2\lambda _{3}^{\Sigma } \\
&=&S^{\Sigma }-2R_{33}^{\Sigma },
\end{eqnarray*}%
by Proposition \ref{Curv} and (\ref{t-1}) we have

\begin{eqnarray*}
\nu ^{\Sigma } &\geq &S-2R_{33}-c\sqrt{Sf^{-1}} \\
&\geq &S-2\lambda _{3}-c\sqrt{Sf^{-1}}.
\end{eqnarray*}%
Therefore, we conclude that (\ref{t0}) follows if we can show that 
\begin{equation}
\nu \geq -\left( \frac{c}{\ln f}\right) ^{\frac{1}{4}},  \label{t1}
\end{equation}%
where 
\begin{equation*}
\nu :=\lambda _{1}+\lambda _{2}-\lambda _{3}.
\end{equation*}%
We will also denote 
\begin{eqnarray}
\lambda &=&\lambda _{1}+\lambda _{3}-\lambda _{2}  \label{t1'} \\
\mu &=&\lambda _{2}+\lambda _{3}-\lambda _{1}.  \notag
\end{eqnarray}%
Note that $\nu \leq \lambda \leq \mu .$

We now prove (\ref{t1}). By Proposition \ref{Curv} and (\ref{t-1}) we know
that 
\begin{eqnarray}
R_{abcd} &=&\left( R_{ac}g_{bd}-R_{ad}g_{bc}+R_{bd}g_{ac}-R_{bc}g_{ad}\right)
\label{t2} \\
&&-\frac{S}{2}\left( g_{ac}g_{bd}-g_{ad}g_{bc}\right) +O\left( \sqrt{Sf^{-1}}%
\right) .  \notag
\end{eqnarray}

It is easy to see that this implies 
\begin{eqnarray*}
\Delta _{f}R_{ac} &=&R_{ac}-2R_{aicj}R_{ij} \\
&=&R_{ac}-2R_{abcd}R_{bd}+O\left( Sf^{-1}\right) \\
&=&R_{ac}-3SR_{ac}+S^{2}g_{ac}+4R_{ad}R_{dc}-2\left\vert \mathrm{Ric}%
\right\vert ^{2}g_{ac}+O\left( Sf^{-\frac{1}{2}}\right) .
\end{eqnarray*}%
This means that, in the sense of barrier, 
\begin{equation*}
\Delta _{f}\lambda _{3}\geq \lambda _{3}-3S\lambda _{3}+S^{2}+4\lambda
_{3}^{2}-2\left\vert \mathrm{Ric}\right\vert ^{2}-cSf^{-\frac{1}{2}}.
\end{equation*}%
Since $\nu =S-2\lambda _{3}$, it follows that 
\begin{equation*}
\Delta _{f}\nu \leq \nu +6S\lambda _{3}-2S^{2}-8\lambda _{3}^{2}+2\left\vert 
\mathrm{Ric}\right\vert ^{2}+cSf^{-\frac{1}{2}}.
\end{equation*}%
A straightforward computation, using (\ref{t1'}) and (\ref{t-1}), then
implies that 
\begin{equation}
\Delta _{f}\nu \leq \nu -\nu ^{2}-\lambda \mu +cSf^{-\frac{1}{2}}.
\label{t3}
\end{equation}%
Modulo the error term $Sf^{-\frac{1}{2}},$ this is the inequality one gets
for three dimensional shrinking gradient Ricci solitons.

Let 
\begin{equation*}
F:=f-2\ln S.
\end{equation*}%
We have 
\begin{eqnarray*}
&&\Delta _{F}\left( \nu S^{-1}\right) \\
&=&\left( \Delta _{f}\nu \right)S^{-1}+\nu \left( \Delta _{f}S^{-1}\right)
+2\left\langle \nabla \nu ,\nabla S^{-1}\right\rangle +2\left\langle \nabla
\ln S,\nabla \left( \nu S^{-1}\right)\right\rangle \\
&\leq &\left( \nu -\nu ^{2}-\lambda \mu +cSf^{-\frac{1}{2}}\right)
S^{-1}+2\left\langle \nabla \ln S,\nabla \left( \nu S^{-1}\right)
\right\rangle \\
&&+\nu \left( -S^{-1}+2\left\vert \mathrm{Ric}\right\vert
^{2}S^{-2}+2\left\vert \nabla S\right\vert ^{2}S^{-3}\right) \\
&&+2\left\langle \nabla \left( \nu S^{-1}\right) ,\nabla S^{-1}\right\rangle
S+2\left\langle \nabla S,\nabla S^{-1}\right\rangle \left( \nu S^{-1}\right)
\\
&=&-S^{-2}\left( \left( \nu ^{2}+\lambda \mu \right) S-2\left\vert \mathrm{%
Ric}\right\vert ^{2}\nu \right) +cf^{-\frac{1}{2}}.
\end{eqnarray*}%
It is easy to see that%
\begin{equation*}
\left( \nu ^{2}+\lambda \mu \right) S-2\left\vert \mathrm{Ric}\right\vert
^{2}\nu =\lambda ^{2}\left( \mu -\nu \right) +\mu ^{2}\left( \lambda -\nu
\right) +O\left( S^{2}f^{-1}\right) .
\end{equation*}
Hence, the function 
\begin{equation*}
u:=\frac{\nu }{S}
\end{equation*}%
satisfies, in the sense of barrier, the following inequality 
\begin{equation}
\Delta _{F}u\leq -S^{-2}\left( \lambda ^{2}\left( \mu -\nu \right) +\mu
^{2}\left( \lambda -\nu \right) \right) +cf^{-\frac{1}{2}}.  \label{t4}
\end{equation}%
We remark that a function similar to $u$ was used to classify locally
conformally flat shrinking Ricci solitons of arbitrary dimension in \cite%
{ELM}. This function also appears in Hamilton-Ivey pinching estimate for
three dimensional ancient solutions \cite{CLN}.

We want to prove a lower bound for the function $u$ based on (\ref{t4}). For
this, let $R>r_{0}$ be large enough so that $R_{1}:=\ln R\geq r_{0}$.

According to Theorem \ref{Rc_improved}, there exists a constant $c_{0}>0$ so
that 
\begin{equation}
u>-c_{0}\text{ \ on }M.  \label{u}
\end{equation}

Consider the function 
\begin{equation*}
w:=u+kf^{-\varepsilon }+\varepsilon S^{-1},
\end{equation*}%
where 
\begin{eqnarray}
\varepsilon &=&\frac{1}{\sqrt{R_{1}}}  \label{k} \\
k &=&c_{0}\left( R_{1}\right) ^{\varepsilon }.  \notag
\end{eqnarray}%
The constant $c_{0}>0$ in (\ref{k}) is the same as that in (\ref{u}). The
choice of $k$ guarantees that 
\begin{equation}
w>0\text{ \ on \ \ }\partial D\left( R_{1}\right) .  \label{positive}
\end{equation}

On $M\backslash D\left( R_{1}\right) ,$ 
\begin{eqnarray*}
\Delta _{F}f^{-\varepsilon } &=&-\varepsilon \left( \Delta _{f}\left(
f\right) \right) f^{-\varepsilon -1}+\varepsilon \left( \varepsilon
+1\right) \left\vert \nabla f\right\vert ^{2}f^{-\varepsilon
-2}-2\varepsilon \left\langle \nabla \ln S,\nabla f\right\rangle
f^{-\varepsilon -1} \\
&\leq &\varepsilon \left( f-2\right) f^{-\varepsilon -1}+\varepsilon \left(
\varepsilon +1\right) f^{-\varepsilon -1}+2\varepsilon \left\vert \nabla \ln
S\right\vert f^{-\varepsilon -\frac{1}{2}} \\
&\leq &2\varepsilon f^{-\varepsilon },
\end{eqnarray*}%
where in the last line we have used Theorem \ref{Scalar}.

Next, we have that 
\begin{eqnarray*}
\Delta _{F}S^{-1} &=&\Delta _{f}S^{-1}+2\left\langle \nabla \ln S,\nabla
S^{-1}\right\rangle \\
&=&-\left( \Delta _{f}S\right) S^{-2}+2\left\vert \nabla S\right\vert
^{2}S^{-3}-2\left\vert \nabla S\right\vert ^{2}S^{-3} \\
&=&-S^{-1}+2\left\vert \mathrm{Ric}\right\vert ^{2}S^{-2} \\
&\leq &-S^{-1}+c,
\end{eqnarray*}%
where in the last line we have used Theorem \ref{Rc_improved}.

Hence, on $M\backslash D\left( R_{1}\right) , $ 
\begin{equation}
\Delta _{F}w\leq -S^{-2}\left( \lambda ^{2}\left( \mu -\nu \right) +\mu
^{2}\left( \lambda -\nu \right) \right) +2\varepsilon kf^{-\varepsilon
}-\varepsilon S^{-1}+c\varepsilon ,  \label{t5}
\end{equation}%
where we have used the fact that $cf^{-\frac{1}{2}}\leq c\left( R_{1}\right)
^{-\frac{1}{2}}=c\varepsilon $ on $M\backslash D\left( R_{1}\right) $.

Let $\phi \left( t\right) =\frac{R-t}{R}$ on $\left[ 0,R\right] $ and
consider the cutoff function $\phi \left( f\right) $ on $D\left( R\right).$
On $D\left( R\right) \backslash D\left( R_{1}\right)$ we have 
\begin{eqnarray}
\left\vert \nabla \phi \right\vert &=&\frac{\left\vert \nabla f\right\vert }{%
R}\leq \frac{1}{\sqrt{R}}  \label{t6} \\
\Delta _{F}\phi &=&\Delta _{f}\phi +2\left\langle \nabla \ln S,\nabla \phi
\right\rangle  \notag \\
&\geq &\frac{1}{R}\left( f-2\right) -\frac{c}{R}\sqrt{f}  \notag \\
&\geq &\frac{1}{2R}f.  \notag
\end{eqnarray}%
In the second line above we have used Theorem \ref{Scalar}.

Now define the function $G:=\phi ^{2}w$ on $M\backslash D\left( R_{1}\right)
, $ which is positive on $\partial D\left( R_{1}\right) $ by (\ref{positive}%
) and zero on $M\backslash D\left( R\right) $. Let us first assume that $G$
is negative somewhere in $D\left( R\right) \backslash D\left( R_{1}\right) $%
. Then there exists an interior point $q$ of $D\left( R\right) \backslash
D\left( R_{1}\right)$ at which $G$ achieves its minimum. In particular, $%
G\left( q\right) <0$ and $\nu \left( q\right) <0$.

Using (\ref{t6}) it follows that at $q,$%
\begin{eqnarray}
\ \ \ 0 &\leq &\phi ^{2}\Delta _{F}G  \label{t7} \\
&=&\phi ^{4}\Delta _{F}w+G\Delta _{F}\phi ^{2}+2\left\langle \nabla w,\nabla
\phi ^{2}\right\rangle \phi ^{2}  \notag \\
&\leq &-S^{-2}\left( \lambda ^{2}\left( \mu -\nu \right) +\mu ^{2}\left(
\lambda -\nu \right) \right) \phi ^{4}+\left( 2k\varepsilon f^{-\varepsilon
}-\varepsilon S^{-1}+c\varepsilon \ \right) \phi ^{4}  \notag \\
&&+\left( \phi \frac{f}{R}-6\left\vert \nabla \phi \right\vert ^{2}\right) G,
\notag
\end{eqnarray}
where we have used 
\begin{equation*}
2\phi \left( \Delta _{F}\phi \right) G\leq \frac{1}{R}\phi fG\text{ \ at }q.
\end{equation*}

We now discuss two cases.

\textbf{Case 1.} Suppose first that at $q$ we have 
\begin{equation*}
\left( \phi \frac{f}{R}-6\left\vert \nabla \phi \right\vert ^{2}\right)
G\leq 0.
\end{equation*}

Then, we see from (\ref{t7}) that 
\begin{equation}
S^{-2}\left( \lambda ^{2}\left( \mu -\nu \right) +\mu ^{2}\left( \lambda
-\nu \right) \right) \leq 2k\varepsilon f^{-\varepsilon }-\varepsilon
S^{-1}+c\varepsilon \ .  \label{t8}
\end{equation}
In particular, (\ref{t8}) implies that 
\begin{eqnarray*}
\varepsilon S^{-1} &\leq &2k\varepsilon f^{-\varepsilon }+c\varepsilon \\
&\leq &2k\varepsilon \left( R_{1}\right) ^{-\varepsilon }+c\varepsilon \\
&=&\left( 2c_{0}+c\right) \varepsilon .
\end{eqnarray*}%
Here we have used the definition of $k$ in (\ref{k}). This shows that there
exists a constant $c_{1}:=2c_{0}+c>0$ so that $S\left( q\right) \geq \frac{1%
}{c_{1}}>0$. Now (\ref{t8}) implies that 
\begin{eqnarray*}
\lambda ^{2}\left( \mu -\nu \right) +\mu ^{2}\left( \lambda -\nu \right)
&\leq &c\left( 2k\varepsilon f^{-\varepsilon }+c\varepsilon \right) \\
&\leq &c\varepsilon .
\end{eqnarray*}%
Hence, 
\begin{eqnarray}
\lambda ^{2}\left( \mu -\nu \right) &\leq &c\varepsilon  \label{t9} \\
\mu ^{2}\left( \lambda -\nu \right) &\leq &c\varepsilon .  \notag
\end{eqnarray}%
In addition, we know that $S(q)=\mu +\lambda +\nu \geq \frac{1}{c_{1}}>0$,
which implies that $\mu \geq \frac{1}{2c_{1}}$. Therefore, one concludes
from the first inequality in (\ref{t9}) (recall $\nu \left( q\right) <0$)
that $\left\vert \lambda \right\vert \leq c_{2}\sqrt{\varepsilon }$. Using
this in the second inequality of (\ref{t9}), we obtain 
\begin{eqnarray*}
-\nu -c_{2}\sqrt{\varepsilon } &\leq &\lambda -\nu \\
&\leq &\frac{c\varepsilon }{\mu ^{2}} \\
&\leq &c\varepsilon .
\end{eqnarray*}%
This proves that $-\nu \left( q\right) \leq c\sqrt{\varepsilon }$ and

\begin{equation*}
G\left( q\right) \geq -c\sqrt{\varepsilon }
\end{equation*}%
as $S\left( q\right) \geq \frac{1}{c_{1}}>0$. In conclusion, 
\begin{equation}
\inf_{D\left( \frac{R}{2}\right) \backslash D\left( R_{1}\right) }w\geq
4G\left( q\right) \geq -c\sqrt{\varepsilon }.  \label{t10}
\end{equation}

\textbf{Case 2.} Suppose now that at $q$ we have 
\begin{equation*}
\left( \phi \frac{f}{R}-6\left\vert \nabla \phi \right\vert ^{2}\right) G>0.
\end{equation*}%
Since $G\left( q\right) <0,$ we conclude that 
\begin{eqnarray*}
\frac{R-f}{R}\frac{f}{R} &\leq &6\left\vert \nabla \phi \right\vert ^{2} \\
&\leq &\frac{6}{R^{2}}f.
\end{eqnarray*}%
Hence, $f\left( q\right) \geq R-6$ and $\phi \left( q\right) \le \frac{6}{R}$%
. We now conclude that 
\begin{eqnarray*}
\inf_{D\left( \frac{R}{2}\right) \backslash D\left( R_{1}\right) }G &\geq
&G\left( q\right) \\
&=&w\left( q\right) \phi ^{2}\left( q\right) \\
&\geq &-\frac{c}{R^{2}},
\end{eqnarray*}%
where the last line follows from $w\geq u>-c_{0}$. In particular,

\begin{equation}
\inf_{D\left( \frac{R}{2}\right) \backslash D\left( R_{1}\right) }w\geq -%
\frac{c}{R^{2}}\text{.}  \label{t11}
\end{equation}

By (\ref{t10}) and (\ref{t11}) we conclude that if $G$ is negative somewhere
in $D\left( R\right) \backslash D\left( R_{1}\right) $, then on $D\left( 
\frac{R}{2}\right) \backslash D\left( R_{1}\right) $ 
\begin{equation}
\frac{\nu }{S}\geq -\frac{k}{f^{\varepsilon }}-\frac{\varepsilon }{S}-c\sqrt{%
\varepsilon }.  \label{t12}
\end{equation}%
Certainly, the same conclusion holds true if $G$ is non-negative on $D\left(
R\right) \backslash D\left( R_{1}\right) $. Therefore, from (\ref{k}) and (%
\ref{t12}) we see that on $D\left( \frac{R}{2}\right) \backslash D\left(
R_{1}\right) ,$ 
\begin{equation}
\nu \geq -c\left( \frac{R_{1}}{f}\right) ^{\varepsilon }-c\sqrt{\varepsilon }%
,  \label{final}
\end{equation}%
Recall that $\varepsilon =\frac{1}{\sqrt{\ln R}}$. So on $\Sigma \left( 
\frac{R}{2}\right) =\partial D\left( \frac{R}{2}\right) $ we get from (\ref%
{final}) that 
\begin{equation*}
\nu \geq -\left( \frac{c}{\ln R}\right) ^{\frac{1}{4}}.
\end{equation*}%
The constant $c$ depends only on $A$ and $\sup_{D\left( r_{0}\right)
}\left\vert \mathrm{Rm}\right\vert $. Since $R$ is arbitrary, this proves
the result.
\end{proof}

\section{Conical structure}

Our goal in this section is to prove the following theorem. 

\begin{theorem}
\label{cone} Let $\left( M,g,f\right) $ be a complete  four
dimensional shrinking gradient Ricci soliton with scalar curvature
converging to zero at infinity. Then there exists a cone $E_{0}$ such that $%
(M,g)$ is $C^{k}$ asymptotic to $E_{0}$ for all $k.$
\end{theorem}

\begin{proof}
We prove that there exist constants $c,C>0$ so that 
\begin{equation}
c\leq S\,f\leq C\text{ \ \ on }M.  \label{s0}
\end{equation}%
The lower bound was established in \cite{CLY}. Here we use the above
estimates to prove the upper bound.

Using Theorem \ref{Rc_improved} we see that there exists a constant $c_{0}>0$
for which 
\begin{eqnarray}
\Delta _{f}S &=&S-2\left\vert \mathrm{Ric}\right\vert ^{2}  \label{s1} \\
&\geq &S-c_{0}S^{2}.  \notag
\end{eqnarray}%
Using that $\Delta _{f}\left( f\right) =2-f,$ we obtain 
\begin{eqnarray*}
\Delta _{f}\left( f^{-1}\right)  &=&-\Delta _{f}\left( f\right)
f^{-2}+2\left\vert \nabla f\right\vert ^{2}f^{-3} \\
&\leq &\left( f-2\right) f^{-2}+2f^{-2} \\
&=&f^{-1}.
\end{eqnarray*}%
Choose $r_{0}\geq 1$ large enough so that on $M\backslash D\left(
r_{0}\right) $ 
\begin{equation}
S\,<\frac{1}{4c_{0}}  \label{s2}
\end{equation}%
for $c_{0}>0$ the constant in (\ref{s1}) and also so that  $6\left\vert \nabla
f\right\vert ^{2}\geq 4f.$ Then, 
\begin{eqnarray*}
\Delta _{f}\left( f^{-2}\right)  &=&2\left( f-2\right) f^{-3}+6\left\vert
\nabla f\right\vert ^{2}f^{-4} \\
&\geq &2f^{-2}.
\end{eqnarray*}%
Define function 
\begin{equation}
u:=S-af^{-1}+c_{0}a^{2}f^{-2},  \label{s3}
\end{equation}%
where 
\begin{equation*}
a:=\frac{r_{0}}{2c_{0}}.
\end{equation*}%
By the choice of $a$ and (\ref{s2}) it follows that 
\begin{equation}
u<0\text{ \ on \ }\partial D\left( r_{0}\right) .  \label{neg}
\end{equation}%
Indeed, on $\partial D\left( r_{0}\right) $ we have that 
\begin{equation*}
S-af^{-1}+2c_{0}a^{2}f^{-2}<\frac{1}{4c_{0}}-\frac{1}{2c_{0}}+\frac{1}{4c_{0}%
}=0.
\end{equation*}%
Now note that 
\begin{eqnarray*}
\Delta _{f}u &\geq &S-c_{0}S^{2}-af^{-1}+2c_{0}a^{2}f^{-2} \\
&=&u-c_{0}S^{2}+c_{0}a^{2}f^{-2} \\
&=&u-c_{0}\left( S-af^{-1}\right) \left( S+af^{-1}\right)  \\
&\geq &u-c_{0}u\left( S+af^{-1}\right) .
\end{eqnarray*}%
Therefore, on $M\backslash D\left( r_{0}\right) ,$ 
\begin{equation}
\Delta _{f}u\geq u\left( 1-c_{0}S-c_{0}af^{-1}\right) .  \label{s4}
\end{equation}

We now claim that 
\begin{equation}
u\leq cf^{-2}\text{ \ on }M\backslash D\left( r_{0}\right) .  \label{s5}
\end{equation}

To prove this claim, let $\psi \left( t\right) =\frac{R-t}{R}$ on $\left[ 0,R%
\right] $ and $\psi =0$ for $t\geq R$. Define $G:=\psi ^{2}u$ and compute 
\begin{gather}
\Delta _{f}G=\psi ^{2}\Delta _{f}u+u\Delta _{f}\psi ^{2}+2\left\langle
\nabla u,\nabla \psi ^{2}\right\rangle   \label{eq} \\
\geq G\,\left( 1-c_{0}S-c_{0}af^{-1}\right)   \notag \\
+2\psi ^{-1}\left( \Delta _{f}\psi \right) G-6\psi ^{-2}\left\vert \nabla
\psi \right\vert ^{2}G+2\psi ^{-2}\left\langle \nabla G,\nabla \psi
^{2}\right\rangle .  \notag
\end{gather}%
Let $q$ be the maximum point of $G$ on $D\left( R\right) \backslash D\left(
r_{0}\right) $. If $G\left( q\right) \leq 0,$ then $u\leq 0$ and the claim (%
\ref{s5}) is true. So we may assume $G\left( q\right) >0.$ In this case (\ref%
{neg}) implies that $q$ is an interior point of $D\left( R\right) \backslash
D\left( r_{0}\right) $. At $q$, by the maximum principle and (\ref{eq}), we
have 
\begin{eqnarray}
0 &\geq &1-c_{0}S-c_{0}af^{-1}+2\psi ^{-1}\left( \Delta _{f}\psi \right)
-6\psi ^{-2}\left\vert \nabla \psi \right\vert ^{2}  \label{s6} \\
&>&2\psi ^{-1}\left( \Delta _{f}\psi \right) -6\psi ^{-2}\left\vert \nabla
\psi \right\vert ^{2}.  \notag
\end{eqnarray}%
Since 
\begin{equation*}
\Delta _{f}\psi =\frac{f-2}{R}\geq \frac{f}{2R},
\end{equation*}%
it follows from (\ref{s6}) that 
\begin{equation*}
\frac{f}{R}\psi \leq 6\left\vert \nabla \psi \right\vert ^{2}\leq 6\frac{1}{%
R^{2}}f.
\end{equation*}%
This means that $\psi \left( q\right) \leq \frac{6}{R}.$ Hence,

\begin{eqnarray*}
G\left( q\right)  &=&u\left( q\right) \psi ^{2}\left( q\right)  \\
&\leq &\frac{c}{R^{2}}.
\end{eqnarray*}%
Therefore,%
\begin{eqnarray*}
\frac{1}{4}\sup_{D\left( \frac{R}{2}\right) \backslash D\left( r_{0}\right)
}u &\leq &\sup_{D\left( \frac{R}{2}\right) \backslash D\left( r_{0}\right) }G
\\
&\leq &G\left( q\right)  \\
&\leq &\frac{c}{R^{2}}.
\end{eqnarray*}%
Since $R$ is arbitrary, this again proves the claim (\ref{s5}).

In conclusion, on $M\backslash D\left( r_{0}\right) ,$ 
\begin{equation*}
S-af^{-1}+c_{0}a^{2}f^{-2}\leq cf^{-2}
\end{equation*}%
or $S\,f\leq a+c.$ This proves (\ref{s0}). 

The theorem now follows as in \cite{KW}. Indeed, using Shi's derivative
estimates, one can get corresponding sharp decay estimates for all covariant
derivatives of the curvature. Certainly, this can also be done directly, by working with the elliptic equations
instead of parabolic ones. These curvature estimates  then prove the convergence
to a cone as required by the theorem. We omit these details and refer the reader to \cite{KW}. 
\end{proof}

\section{Diameter estimate}

In this section, we work with compact shrinking gradient Ricci solitons of
arbitrary dimension and establish a diameter estimate from above. Namely, we
prove the following.

\begin{theorem}
Let $\left( M,g,f\right) $ be a compact gradient shrinking Ricci soliton of
dimension $n$. Then the diameter of $\left( M,g\right) $ has an upper bound
of the form 
\begin{equation*}
\mathrm{diam}\left( M\right) \leq c\left( n,\mathrm{inj}\left( M\right)
\right) ,
\end{equation*}%
where $\mathrm{inj}\left( M\right) $ is the injectivity radius of $\left(
M,g\right) .$
\end{theorem}

\begin{proof}
For simplicity, we will henceforth assume that $\mathrm{inj}\left( M\right)
=2,$ and prove that $\mathrm{diam}\left( M\right) \leq c\left( n\right) .$
Since $M$ is compact, the potential $f$ assumes a maximum and a minimum
value. Let us fix 
\begin{eqnarray*}
f\left( p\right) &=&\min_{M}f \\
f\left( q\right) &=&\max_{M}f.
\end{eqnarray*}%
We continue to normalize $f$ so that 
\begin{equation}
S+\left\vert \nabla f\right\vert ^{2}-f=0,  \label{0}
\end{equation}%
where $S$ is the scalar curvature. Recall again that 
\begin{eqnarray}
f\left( x\right) &\leq &\left( \frac{1}{2}d\left( p,x\right) +c\left(
n\right) \right) ^{2}\text{ \ for all }x\in M,  \label{1} \\
f\left( x\right) &\geq &\frac{1}{4}d^{2}\left( p,x\right) -c\left( n\right)
d\left( p,x\right) \text{ for all }x\in M\backslash B_{p}\left( r_{0}\left(
n\right) \right) .  \notag
\end{eqnarray}%
Here both $c\left( n\right) $ and $r_{0}\left( n\right) $ depend only on
dimension $n.$ Since $S\geq 0,$ we see that (\ref{1}) provides a uniform
upper bound estimate for $\left\vert \nabla f\right\vert $ as well. Indeed, $%
\left\vert \nabla f\right\vert ^{2}\leq f$.

Consider now a minimizing normal geodesic $\sigma$ joining $p$ and $q,$
parametrized so that $\sigma \left( 0\right) =p$ and $\sigma \left( R\right)
=q.$ We apply the second variation formula of arc length to $\sigma \left(
s\right) ,$ $0\leq s\leq R,$ and obtain

\begin{equation*}
\int_{0}^{R}\mathrm{Ric}\left( \sigma ^{\prime }\left( s\right) ,\sigma
^{\prime }\left( s\right) \right) \phi ^{2}\left( s\right) ds\leq \left(
n-1\right) \int_{0}^{R}\left( \phi ^{\prime }\left( s\right) \right) ^{2}ds
\end{equation*}%
for any Lipschitz function $\phi $ with compact support in $\left[ 0,R\right]%
.$ Using the fact that

\begin{equation*}
\mathrm{Ric}(\sigma ^{\prime }\left( s\right) ,\sigma^{\prime }\left(
s\right))+f^{\prime \prime }(s)=\frac{1}{2}
\end{equation*}
and integrating by parts, we obtain

\begin{equation}
\frac{1}{2}\int_{0}^{R}\phi ^{2}\left( s\right) ds\leq \left( n-1\right)
\int_{0}^{R}\left( \phi ^{\prime }\left( s\right) \right)
^{2}ds-2\int_{0}^{R}f^{\prime }\left( s\right) \phi \left( s\right) \phi
^{\prime }\left( s\right) ds,  \label{2}
\end{equation}%
where 
\begin{equation}
f\left( s\right) :=f\left( \sigma \left( s\right) \right) .  \label{a}
\end{equation}

For any $R-1\leq t\leq R-\frac{1}{2},$ let us take 
\begin{equation*}
\phi \left( s\right) :=\left\{ 
\begin{array}{c}
s \\ 
1 \\ 
\frac{R-s}{R-t}%
\end{array}%
\right. 
\begin{array}{c}
\text{for }0\leq s\leq 1 \\ 
\text{for }1\leq s\leq t \\ 
\text{for }t\leq s\leq R%
\end{array}%
\end{equation*}%
Then we get from (\ref{2}) that 
\begin{eqnarray*}
\frac{1}{2}\left( t-1\right) &\leq &\frac{1}{2}\int_{0}^{R}\phi ^{2}\left(
s\right) ds \\
&\leq &\left( n-1\right) \int_{0}^{R}\left( \phi ^{\prime }\left( s\right)
\right) ^{2}ds-2\int_{0}^{R}f^{\prime }\left( s\right) \phi \left( s\right)
\phi ^{\prime }\left( s\right) ds \\
&=&\left( n-1\right) \left( 1+\frac{1}{R-t}\right) -2\int_{0}^{1}f^{\prime
}\left( s\right) sds+\frac{2}{\left( R-t\right) ^{2}}\int_{t}^{R}f^{\prime
}\left( s\right) \left( R-s\right) ds.
\end{eqnarray*}%
By (\ref{1}) and the subsequent comments, it is easy to see that 
\begin{equation*}
\sup_{B_{p}\left( 1\right) }\left\vert \nabla f\right\vert \leq c\left(
n\right) .
\end{equation*}%
This implies that 
\begin{equation*}
\int_{t}^{R}f^{\prime }\left( s\right) \left( R-s\right) ds\geq \frac{1}{16}%
R-c\left( n\right) .
\end{equation*}%
Integration by parts then yields

\begin{equation*}
-\left( R-t\right) f\left( t\right) +\int_{t}^{R}f\left( s\right) ds\geq 
\frac{1}{16}R-c\left( n\right) .
\end{equation*}%
Since $f\left( s\right) \leq f\left( R\right) =\max f$ and $\frac{1}{2}\leq
R-t\leq 1,$ we see that

\begin{eqnarray*}
-\left( R-t\right) f\left( t\right) +\int_{t}^{R}f\left( s\right) ds &\leq
&\left( R-t\right) \left( f\left( R\right) -f\left( t\right) \right) \\
&\leq &f\left( R\right) -f\left( t\right) .
\end{eqnarray*}%
Thus,

\begin{equation}
f\left( R\right) -f\left( t\right) \geq \frac{1}{16}R-c\left( n\right)
\label{3}
\end{equation}%
for all $R-1\leq t\leq R-\frac{1}{2}.$

Now the assumption that $\mathrm{inj}\left( M\right) =2$ implies that the
geodesic $\sigma \left( s\right) ,$ $R-1\leq s\leq R,$ can be extended into
a minimizing normal geodesic over $R-1\leq s\leq R+1.$

We consider the cutoff function $\psi $ on $\left[ R-1,R+1\right] $ defined
by 
\begin{equation*}
\psi \left( t\right) :=\left\{ 
\begin{array}{c}
t-\left( R-1\right) \\ 
R-t%
\end{array}%
\right. 
\begin{array}{l}
\text{for }R-1\leq t\leq R \\ 
\text{for }R\leq t\leq R+1%
\end{array}%
\end{equation*}

Applying the second variation formula to $\sigma \left( t\right) $ for $%
R-1\leq t\leq R+1,$ we have that (see (\ref{2})) 
\begin{equation*}
\frac{1}{2}\int_{R-1}^{R+1}\psi ^{2}\left( t\right) dt\leq \left( n-1\right)
\int_{R-1}^{R+1}\left( \psi ^{\prime }\left( t\right) \right)
^{2}dt-2\int_{R-1}^{R+1}f^{\prime }\left( t\right) \phi \left( t\right) \phi
^{\prime }\left( t\right) dt.
\end{equation*}%
This implies 
\begin{equation*}
\int_{R-1}^{R}f^{\prime }\left( t\right) \phi \left( t\right)
dt-\int_{R}^{R+1}f^{\prime }\left( t\right) \phi \left( t\right) dt\leq
c\left( n\right) .
\end{equation*}

After integrating by parts, this can be rewritten as 
\begin{equation}
2f\left( R\right) \leq \int_{R-1}^{R+1}f\left( t\right) dt+c\left( n\right) .
\label{4}
\end{equation}%
Note that $f\left( R\right) =f\left( q\right) =\max f.$ So 
\begin{equation*}
\int_{R-1}^{R+1}f\left( t\right) dt\leq \int_{R-1}^{R-\frac{1}{2}}f\left(
t\right) dt+\frac{3}{2}f\left( R\right) .
\end{equation*}%
By (\ref{4}), this implies 
\begin{equation}
f\left( R\right) \leq 2\int_{R-1}^{R-\frac{1}{2}}f\left( t\right) dt+c\left(
n\right) .  \label{5}
\end{equation}%
Using (\ref{3}), we conclude 
\begin{equation*}
d\left( p,q\right) =R\leq c\left( n\right) .
\end{equation*}%
Therefore, by (\ref{1}), 
\begin{equation*}
\max_{M}f\leq c\left( n\right) .
\end{equation*}%
Now the lower bound of $f$ from (\ref{1}) implies 
\begin{equation*}
d\left( p,x\right) \leq c\left( n\right)
\end{equation*}%
for all $x\in M.$ By the triangle inequality, one immediately sees that 
\begin{equation*}
\mathrm{diam}\left( M\right) \leq c\left( n\right) .
\end{equation*}%
This proves the theorem.
\end{proof}

\section{Acknowledgment}

We wish to thank Huai-Dong Cao for discussions that have motivated this
work. The first author was partially supported by NSF grant No. DMS-1262140
and the second author by NSF grant No. DMS-1105799.

\address {\noindent Department of Mathematics\\ University of Connecticut\\
Storrs, CT 06268\\ USA\\ \email{\textit{E-mail address}: {\tt
ovidiu.munteanu@uconn.edu} } \vskip 0.3in \address {\noindent School of
Mathematics \\ University of Minnesota\\ Minneapolis, MN 55455\\ USA\\
\email{\textit{E-mail address}: {\tt jiaping@math.umn.edu}} \end{document}